\documentclass[12pt]{amsart}
\usepackage{amssymb,amscd,amsmath}
\usepackage{verbatim} %For long comments
\usepackage{amsfonts}
\usepackage{fullpage}
\usepackage{array}
\usepackage{amsthm}
\usepackage{url}
\usepackage{mathrsfs}
\newtheorem{theorem}{Theorem}[section]
\newtheorem{proposition}[theorem]{Proposition}
\newtheorem{lemma}[theorem]{Lemma}
\newtheorem{corollary}[theorem]{Corollary}

\newtheorem{conjecture}[theorem]{Conjecture}
\theoremstyle{definition}
\newtheorem{definition}[theorem]{Definition}
\newtheorem{remark}[theorem]{Remark}

\newtheorem*{remark*}{Remark}
\newtheorem*{proposition*}{Proposition}
\newtheorem{example}[theorem]{Example}
\newcommand{\GL}{\text{GL}}
\newcommand{\Spec}{\text{Spec}}
\input xy
\xyoption{all}
\title{Zeta Integrals on Arithmetic Surfaces}
\author{Thomas Oliver}
\thanks{Heilbronn Institute for Mathematical Research, University of Bristol, Bristol, UK.}
\date{\today}
\begin{document}
\maketitle
%\subsection*{Dedication} To Professor S.V. Vostokov on the occasion of his 70th birthday.
\subsection*{Abstract}
Given a (smooth, projective, geometrically connected) curve over a number field, one expects its Hasse--Weil $L$-function, a priori defined only on a right half-plane, to admit meromorphic continuation to $\mathbb{C}$ and satisfy a simple functional equation. Aside from exceptional circumstances, these analytic properties remain largely conjectural. One may formulate these conjectures in terms of zeta functions of two-dimensional arithmetic schemes, on which one has non-locally compact ``analytic'' adelic structures admitting a form of ``lifted'' harmonic analysis first defined by Fesenko for elliptic curves. In this paper we generalize his global results to certain curves of arbitrary genus by invoking a renormalizing factor which may be interpreted as the zeta function of a relative projective line. We are lead to a new interpretation of the ``gamma factor'' (defined in terms of the Hodge structures at archimedean places) and an (two-dimensional) adelic interpretation of the ``mean-periodicity correspondence'', which is comparable to the conjectural automorphicity of Hasse--Weil $L$-functions.
\tableofcontents
\section{Introduction}
Let $\mathcal{S}$ be a scheme of finite type over $\mathbb{Z}$. The zeta function of $\mathcal{S}$ is defined on $\Re(s)>\text{dim}(\mathcal{S})$ by the Euler product
\[\zeta(\mathcal{S},s)=\prod_{x\in|\mathcal{S}|}\frac{1}{1-|k(x)|^{-s}},\]
where $|\mathcal{S}|$ denotes the atomisation of $\mathcal{S}$, that is, its set of closed points, and $k(x)$ is the residue field at a closed point $x$. The question of meromorphic continuation of such zeta functions remains open, along with the conjectural functional equation with respect to $s\mapsto\dim({\mathcal{S}})-s$.
\newline A basic case is when $\mathcal{S}$ is a proper, regular model of a smooth projective curve $C$ over a number field $k$. Associated to each \'{e}tale cohomology group of $C$, one has a Hasse--Weil $L$-function $L(H^i(C),s)$ defined as the Euler product of reciprocal characteristic polynomials of the action of Frobenius on the inertia invariants. The cases $i=0,2$ reduce to Dedekind zeta functions of the ground field $k$, and so the unknown quantity is $L(C,s):=L(H^1(C),s)$. One has the following equation
\[\zeta(\mathcal{S},s)=n(\mathcal{S},s)\frac{\zeta_k(s)\zeta_k(s-1)}{L(C,s)}.\]
This expression may be viewed as a definition of $n(\mathcal{S},s)$, which is easily seen as a product of explicit functions rational in a variable of the form $p^{-s}$, where $p$ ranges over the residual characteristics of bad reduction. From this one infers that the meromorphic continuation of $\zeta(\mathcal{S},s)$ is equivalent to that of $L(C,s)$. An exercise in \'{e}tale cohomology demonstrates that $n(\mathcal{S},s)$ admits the correct functional equation so that those of $L(C,s)$ and $\zeta(\mathcal{S},s)$ are equivalent \cite{DRCACOC}.
\newline By taking into account the additional contribution of finitely many horizontal curves on $\mathcal{S}$, in this paper we study a modified zeta function
\[Z(\mathcal{S},\{k_i\},s)=\zeta(\mathcal{S},s)\prod_{i=1}^n\zeta(k_i,s/2),\]
where the number fields $k_i/k$ are determined by the horizontal curves. Up to sign, the functional equation of $L(C,s)$ is equivalent to
\[\mathcal{Z}(\mathcal{S},\{k_i\},s)^2=\mathcal{Z}(\mathcal{S},\{k_i\},2-s)^2,\]
where $\mathcal{Z}$ is not quite the product of the completions:
\[\mathcal{Z}(\mathcal{S},\{k_i\},s)=Z(\mathcal{S},\{k_i\},s)A(\mathcal{S})^{(1-s)/2}\Gamma(\mathcal{S},s)\prod_{i=1}^n\Gamma(k_i,s/2).\]
In the above expression $A(\mathcal{S})$ denotes the conductor of $\mathcal{S}$ and $\Gamma(\mathcal{S},s)$ (resp. $\Gamma(k_i,s)$) denotes the gamma factor of $\mathcal{S}$ (resp. $k_i$). These quantities will be defined in the main body of this text. 
\newline For number fields, or curves over finite fields, it is well understood that the analytic properties of zeta functions can be obtained through harmonic analysis on a commutative adelic group - this is reviewed in section~\ref{tate.section}. These techniques have long since been extended to various non-commutative algebraic groups. Our goal is to develop this idea on certain two-dimensional adelic groups. We will review the theory of two dimensional local fields in section $4.3$, and two-dimensional analytic adeles in $4.4$. The most fundamental issue is that these groups are not locally compact, and so what we mean by ``harmonic analysis'' has to be somewhat modified. The development of harmonic analysis on more general topological groups is of the utmost importance, as mentioned as far back as Weil~\cite[Foreword]{BNT}.
\newline The approach taken here is to allow our measure to take values not in $\mathbb{R}$, rather in the field of Laurent series $\mathbb{R}((X))$, which is itself a two-dimensional local field. We follow the approach of Fesenko \cite{F1, F3, F4}, though we note similarities to techniques of motivic integration which we will not seek to expose here.  
\newline Following a sketch given in \cite[Section~57]{F3}, section~\ref{zetaintegral.section} introduces zeta integrals extending those of Fesenko in the case where $C=E$ is an elliptic curve, the primary difference being a renormalising factor whose arithmetic interpretation is a power of $\zeta(\mathbb{P}^1(\mathcal{O}_k),s)$. Fesenko's original definition of zeta integrals diverges for higher genus curves, due to a certain incompatibility of the additive and multiplicative measures, which is rectified by the renormalizing factor. Moreover, there is a simple connection between this factor and the archimedean components (a quotient of gamma functions depending on Hodge structure) of the completed zeta function, as will be explained.
\newline When $C$ has simple reduction properties (which can always be obtained after base change), we will show that $\mathcal{Z}(\mathcal{S},\{k_i\},s)^2$ is an integral over certain two dimensional ``analytic adeles'', up to the square of a rational function $Q(s)$ of the following form
\[Q(s)=\frac{D}{(1-s)^{m}},\]
for constants $D\in\mathbb{C}$ and $m\in\mathbb{N}$, each of which depend on the base field $k$ and will be explicitly given. Therefore $Q(s)^2$ is invariant with respect to $s\mapsto2-s$. The required reduction properties for the integral expressions are explained in~\ref{good.section}, and are broadly comparable to semistablity. When $C$ possesses these reduction properties, the conductor $A(\mathcal{S})$ arises from counting singularities on bad fibres.
\newline The Hasse--Weil $L$-function is expected to be automorphic. This expectation is not held for the zeta function $\zeta(\mathcal{S},s)$, which, as a proper quotient of (conjecturally) automorphic $L$-functions, is not in the Selberg class. One possible replacement is the notion of mean-periodicity, as studied extensively in \cite{SRF}. We will state the relevant conjecture in $7.1$ and conclude this paper with an adelic interpretation of the mean-periodicity condition. This is the first step towards its verification via adelic duality.
\begin{remark}
The fact that it is the square of the modified completed zeta function appearing in the integral expressions is due to the fact we integrate over two-copies of the multiplicative group of the analytic adeles. Of course, one might expect that integrating over a single copy would give rise to the completed zeta function itself - this is not true at finitely many factors. There are two further reasons for considering the square of the zeta function. Firstly, in this way we avoid issues with the sign of the functional equation. Secondly, there is a compatibility with two-dimensional class field theory, which is the basis of a $\GL_1(\mathbb{A}(\mathcal{S}))$-theory that will not be included in this work.  We also observe that the integration theory used in the work of Fesenko and current paper plays a role in the theory of algebraic groups over fields of the form $\mathbb{Q}_p((t))$, for example, we have the application in \cite{SHAOSL2O2DLF}.
\end{remark}
\section{Tate's Thesis}\label{tate.section}
We will begin by summarizing the content of Tate's thesis for Dedekind zeta functions \cite{T1}. Let $k$ be a number field with ring of integers $\mathcal{O}_k$. The Dedekind zeta function of $\mathcal{O}_k$ is then the zeta function of the arithmetic scheme $\mathcal{S}=\text{Spec}(\mathcal{O}_k)$:
\[\zeta(\mathcal{S},s)=\zeta(k,s).\]
Associated to $k$ one has the locally compact groups of adeles $\mathbb{A}_k$ and ideles $\mathbb{A}_k^{\times}$. $k^{\times}$ is embedded diagonally into $\mathbb{A}_k^{\times}$ and the module map $|~|:\mathbb{A}_k^{\times}\rightarrow\mathbb{R}^{\times}_{+}$ is such that for $\alpha\in k^{\times}$, $|\alpha|=1$. 
\newline Let $f\in S(\mathbb{A}_k)$, the adelic Schwartz space, be defined as follows
\[f(\alpha)=\otimes_vf_v(\alpha_v)\]
\[f_v(\alpha_v)=\begin{cases}\text{char}(\mathcal{O}_v)(\alpha_v),&v\nmid\infty \\
\exp(-\pi \alpha_v^2), & v\text{ real,}\\
\exp(-2\pi|\alpha_v|^2), & v\text{ complex,}
\end{cases}\]
where in the case of a non-archimedean place $v\nmid\infty$, $\mathcal{O}_v$ is the ring of integers of the completion $k_v$. For all $s>1$, the following integral absolutely converges
\[\zeta(f,s):=\int_{\mathbb{A}_k^{\times}}f(\alpha)|\alpha|^sd\mu(\alpha)=\xi(k,s),\]
where $\mu$ is a Haar measure on $\mathbb{A}_k^{\times}$. In fact, up to scalar multiplication there is a unique such measure. Such integrals were known and studied by Artin, Weil, Iwasawa, Tate and many other mathematicians, and we will call them ``one-dimensional zeta integrals.''.
\newline In order to proceed, one applies basic techniques of integration (the Fubini property) and harmonic analysis on the locally compact multiplicative group of ideles. Adelic duality, whose incarnation is the theta formula and Riemann-Roch theorem, then implies the analytic continuation and functional equation of the zeta-function. More precisely, one shows that there exists an entire function $\eta_f(s)$ such that
\[\xi(k,s)=\eta_f(s)+\eta_{\widehat{f}}(1-s)+\omega_f(s),\]
where $\widehat{f}\in S(\mathbb{A}_k)$ is the Fourier transform of $f$, and $\omega_f:\mathbb{C}\rightarrow\mathbb{C}$ is the Laplace transform of a rational function:
\[\omega_f(s)=\int_0^1h_f(x)x^s\frac{dx}{x},\]
\[h_f(x)=-\mu(\mathbb{A}_k^1/k^{\times})(f(0)-x^{-1}\widehat{f}(0)).\]
The explicit form of $h_f(x)$ clearly implies the meromorphic continuation and functional equation of $\xi(k,s)$. The function $h_{f}(x)$ is closely related an integral over the weak topological boundary of a global subspace of the adeles. More precisely, if $\mathbb{A}^1_k$ denotes the set of ideles of norm $1$, then
\[h_f(x):=-\int_{\gamma\in\mathbb{A}^1_k/k^{\times}}\int_{\beta\in\partial k^{\times}}(f(x\gamma\beta)-x^{-1}\widehat{f}(x^{-1}\gamma\beta))d\mu(\beta)d\mu(\gamma).\]
The boundary $\partial k^{\times}$ is that of the multiplicative group $k^{\times}$ with respect to the weak (or ``initial'') topology on $\mathbb{A}_k$, which is simply $k\backslash k^{\times}=\{0\}$. A definition of this topology may be found in \cite[I,~2.3]{GT}.
\newline More generally, one can consider one-dimensional zeta integrals where $|~|^s$ is replaced by an arbitrary quasi-character $\chi$ of the multiplicative group of ideles. In this setting one deduces the basic analytic properties of Hecke $L$-functions. 
\newline Fesenko attempted to extend these ideas to dimension $2$ as follows \cite{F3, F4}. Let $\mathcal{E}$ be a proper, regular model of an elliptic curve $E$ over a number field $k$, then \cite[Section~3]{F3} shows that there is an entire function $\eta_{\mathcal{E}}$ such that
\[A(\mathcal{E})^{1-s}\zeta(\mathcal{E},s)^2=\eta_{\mathcal{E}}(s)+\eta_{\mathcal{E}}(2-s)+\omega_{\mathcal{E}}(s),\]
where $\omega_{\mathcal{E}}(s)$ is defined for $\Re(s)>2$. Generalizing the embedding $k\hookrightarrow\mathbb{A}_k$, there is a semi-global ring of adeles\footnote{We will see the definition of this ring in section 7.} $\mathbb{B}(\mathcal{E})\hookrightarrow\mathbb{A}(\mathcal{E})$ such that, with respect to an inductive limit of weak topologies for a given family of characters, $\omega_{\mathcal{E}}(s)$ is closely related to an integral over the boundary of $\mathbb{B}(\mathcal{E})$. We will develop these ideas for higher genus curves.
\section{Conventions}\label{good.section}
Let $C$ be a smooth, projective and geometrically connected curve over a number field $k$. We now specify a model $\mathcal{S}$ of $C$ simplistic enough for application of two dimensional adelic analysis in its current form. A further development of the theory of lifted harmonic analysis should allow for application to a more general class of arithmetic surfaces. If one is willing to base-change, no restrictions are required\footnote{One may be able to proceed by considering zeta integrals twisted by Galois characters and attempting a descent to the base field. We will not do this here, as much remains to be developed.}. 
\newline Let $\mathcal{B}$ be a Dedekind scheme of dimension $1$, and let $\pi:\mathcal{S}\rightarrow\mathcal{B}$ be a regular, integral, projective, flat two-dimensional $\mathcal{B}$-scheme. We will call such an $\mathcal{S}$ an arithmetic surface. Closed, irreducible curves on $\mathcal{S}$ are either horizontal or vertical. More precisely, such curves are either an irreducible component of a special fibre or the closure of a closed point of the generic fibre, the latter being finite and surjective onto the base $\mathcal{B}$.
\newline Since $\mathcal{S}$ is regular, the special fibre $\mathcal{S}_b$ over a closed point $b\in\mathcal{B}$ is the Cartier divisor $\pi^{\ast}b$. If a given special fibre $\mathcal{S}_b$ contains $r$ irreducible components $\mathcal{S}_{b,i}$, with multiplicity $d_i$, then, as Weil divisors
\[\mathcal{S}_b=\sum_{1\leq i\leq r}d_i\mathcal{S}_{b,i}.\]
An effective divisor $D$ on a regular Noetherian scheme $X$ is said to have normal crossings if, at each point $x\in X$, there exist a system of parameters $f_1,\dots,f_n$ of $X$ at $x$ such that, for some positive integer $m\leq n$, there are integers $r_1,\dots,r_m$ such that $\mathcal{O}_X(-D)_x$ is generated by $f_1^{r_1}\dots f_m^{r_m}$. If $D=\mathcal{S}_b$ is the fibre over $b\in\mathcal{B}$, below we will ask for this property to be true over the residue field $k(b)$, in short, we will be asking for split singularities.
\newline The zeta function depends only on the atomization of $\mathcal{S}$, in particular the zeta function agrees with that of the reduced part $\mathcal{S}_{\text{red}}$. With that in mind, for the purposes of adelic analysis we will only work with the reduced part of each fibre, tacitly using the same notation:
\[\mathcal{S}_b:=\sum_{1\leq i\leq r}\mathcal{S}_{b,i}.\]
On finitely many reduced fibres $\mathcal{S}_b$, there may well be non-smooth points. In this section we will only work with ordinary double points. Moreover, we need $S_{b}$ to be a normal crossing divisor over $k(b)$, so we will assume the ordinary double points are split. To summarize: 
\newline\textit{We will assume that the reduced part of each fibre on $\mathcal{S}$ has only split ordinary double points.}
\newline From now on, $\mathcal{B}$ will be $\text{Spec}(\mathcal{O}_k)$, where $k$ is a number field. Let $C$ be a smooth, projective geometrically irreducible curve of genus $g$ over $k$, such that $C$ has good reduction in all residual characteristics less than $2g+1$. This ensures that the Swan character is trivial and the conductor of $\mathcal{S}$ can be computed by counting singularities as in \cite{CDATNFOAS}.
\begin{remark}\label{semistable.remark}
Some authors describe an arithmetic surface $\mathcal{S}\rightarrow\text{Spec}(\mathcal{O}_k)$ as semistable if $\mathcal{S}$ is a regular $\text{Spec}(\mathcal{O}_k)$-curve with smooth generic fibre and all closed fibres are reduced normal crossing divisors.
\end{remark}
\begin{remark}
One could work with a more general class of curves by incorporating extensions of the base field. More precisely, let $C$ be a smooth, projective geometrically connected curve of genus $\geq 2$ over the function field $K$ of a one-dimensional Dedekind scheme $\mathcal{B}$. By the Deligne--Mumford theorem (\cite{TIOTSOCOGG}, \cite[Theorem~10.4.3]{L}), there exists a Dedekind scheme $\mathcal{B}'$ with function field $K'$ such that the extension $C_{K'}$ has a unique stable model over $S'$. One can take the extension $K'/K$ to be separable. This base change required by the Deligne--Mumford theorem is not intractable. Let $G:=\text{Gal}(K'/K)$, which has a natural action on $\mathcal{S}$, lifting that on $\text{Spec}(\mathcal{O}_{K'})$. The stable reduction, along with its natural $G$-action determines the local factors of the $L$-function (for example, see~\cite[Theorem~1.1]{CLFASROSC}).
\end{remark}
We conclude this section with some notation to be used throughout this paper.
\subsubsection*{Notation} The function field of $\mathcal{S}$ is denoted by $K$. Closed points of $\mathcal{S}$ are denoted $x$, and $y$ will denote an irreducible fibre or horizontal curve. When $y$ is an irreducible component of a fibre, its genus is denoted $g_y$ and function field $k(y)$. The maximal finite subfield of $k(y)$ has cardinality denoted $q(y)$. The set of components of a fibre $\mathcal{S}_{\mathfrak{p}}$ is denoted by $\text{comp}(\mathcal{S}_{\mathfrak{p}})$. If $x$ is a singular point on a fibre $\mathcal{S}_{\mathfrak{p}}$, then
\[\mathcal{S}_{\mathfrak{p}}(x)=\cup_{y\in\text{comp}(\mathcal{S}_{\mathfrak{p}})}y(x),\]
where $y(x)$ denotes the set of local branches of $y$ at $x$.
\section{Two-Dimensional Local Fields}\label{higherhaar.subsection}
Let $\mathcal{S}$ be a two dimensional, irreducible, Noetherian scheme and let $x\in y\subset\mathcal{S}$ be a complete flag of irreducible closed subschemes. If $\mathfrak{m}$ is a local equation for $x$ and $\mathfrak{p}$ is a local equation for $y$, then let $\mathcal{O}=\widehat{\mathcal{O}}_{\mathcal{S},x}$ and
\[K_{x,y}=\text{Frac}(\widehat{(\mathcal{O})_{\mathfrak{p}\mathcal{O}}}),\]
see, for example, \cite{RAA}, \cite{P3}, \cite{IHLF}, \cite[Part~1]{GAATRRFOCOS} and \cite[Sections~6,~7]{AITHDLFAA}.
\newline If $x$ is a smooth point of $y$, then $K_{x,y}$ is an example of a two-dimensional local field; it is a complete discrete valuation field whose residue field is a one-dimensional local field. If $x$ is a singular point on $y$, the same construction yields a direct product of two-dimensional local fields. Recall that $y(x)$ is the set of local branches of $y$ at $x$, then
\[K_{x,y}=\prod_{z\in y(x)}K_{x,z},\]
where $K_{x,z}$ is the two-dimensional local field associated to $x$ and the minimal prime $z$.
\newline The residue field of $K_{x,z}$ will always be denoted $E_{x,z}$. A lift of a local parameter from $E_{x,z}$ to $K_{x,z}$ will be denoted $t_{1,x,z}$, and the cardinality of the residue field of $E_{x,z}$ (which is the second residue field of $K_{x,z}$) is denoted $q(x,z)$.
\newline Let $F$ be a two-dimensional local field. As a complete discrete valuation field, it has the a discrete valuation
\[v_2:F\twoheadrightarrow\mathbb{Z},\]
for which we fix a local parameter $t_2$. We denote the ring of integers with respect to this valuation $\mathcal{O}_F$.
\newline On the residue field $\overline{F}$ we have the discrete valuation
\[v_1:\overline{F}\twoheadrightarrow\mathbb{Z}.\]
Together, $v_1$ and $v_2$ induce a ``rank 2'' valuation on $F$, which depends on $t_2$:
\[\underline{v}:F\rightarrow\mathbb{Z}^2\]
\[\alpha\mapsto(v_1(\alpha t_2^{-v_2(\alpha)}),v_2(\alpha)),\]
where $\mathbb{Z}^2$ is given the lexicographic ordering. Let $O_F$ denote the ring of integers with respect to $\underline{v}$, we have:
\[O_F=\{x\in\mathcal{O}_F:\overline{x}\in\mathcal{O}_F\}.\]
Unlike the classical situation, there are infinitely many different rank $2$ discrete valuations on $F$, however, the ring of integers and maximal ideal do not depend on this choice.
\newline When $F=K_{x,y}$ we use the notations:
\[\mathcal{O}_F=\mathcal{O}_{x,y},\]
\[O_F=O_{x,y},\]
and when $y=\mathcal{S}_{\mathfrak{p}}$ is the fibre of $\mathcal{S}$ over $\mathfrak{p}\in\Spec(\mathcal{O}_k)$ we will write
\[\mathcal{O}_{x,\mathfrak{p}}:=\mathcal{O}_{x,\mathcal{S}_{\mathfrak{p}}},\]
\[O_{x,\mathfrak{p}}:=O_{x,\mathcal{S}_{\mathfrak{p}}}\]
It is well known that complete discrete valuation fields have a non-trivial $\mathbb{R}$-valued Haar measure only when their residue field is finite. In particular, there is no $\mathbb{R}$-valued Haar measure on higher dimensional local fields. A lifted $\mathbb{R}((X))$-valued Haar measure and integration theory appeared in \cite{F2,F3}. In these papers Fesenko develops two approaches to the theory of higher Haar measure on higher local fields, taking values in formal power series over $\mathbb{R}$. A third, lifting approach, suggested in \cite{F3} was further developed by Morrow in \cite{Mor2}. All these approaches give essentially the same translation invariant measure on a class of measurable subsets of $F$. There is also a model-theoretic approach of Hrushovski-Kazdhan \cite{IIVF}.
\begin{example}\label{localtheory.example}
Let $F$ be a two-dimensional local field, with a fixed local parameter $t_2$ and residue field $K$. On the locally compact field $K$ we have a Haar measure $\mu_K$, normalized so that $\mu_K(\mathcal{O}_K)=1$. Let $\mathcal{A}$ be the minimal ring of sets generated by $\alpha+t_2^jp^{-1}(S)$, where $S$ is $\mu_K$-measurable, the ``measure'' of a generator of $\mathcal{A}$ is $X^i\mu_K(S)\in\mathbb{R}((X))$. For example $\mu_F(O_F)=1$, where $O_F$ is the rank two ring of integers. This measure extends to a well-defined additive function on $\mathcal{A}$, which is moreover, countably additive in a certain refined sense, \cite[Part~6]{F1}, \cite{F2}, \cite{F3}.
\end{example}
We observe the following:
\begin{enumerate}
\item Essential role was played by a choice of a local parameter $t_2$. An analogous statement will be true in the adelic counterpart of example~\ref{adelictheory.example}.
\item In the mixed characteristic case there are non-linear changes of variables for which the Fubini property of the measure does not hold \cite{FTANLCOVOATDLF}. This could be considered as an example of the non-commutativity inherent in studying $L$-functions of curves over global fields. In this paper, such considerations will not cause a problem.
\end{enumerate}
\section{Analytic Adeles}\label{analadele.subsection}
Let $X$ be a Noetherian scheme, let $M$ be a quasi-coherent sheaf on $X$, and let $T$ be a set of reduced chains on $X$. To such a triple $(X,M,T)$, one can associate an abelian group $\textbf{A}(X,M,T)$ of adeles. We will call these groups ``geometric adeles'' and recommend the following references for details \cite{P1}, \cite{P3}, \cite{RAA}, \cite{OTPBAFS}, \cite{GAATRRFOCOS} and \cite[Section~8]{AITHDLFAA}.
\newline The adelic group $\textbf{A}(X,M,T)$ can be interpreted as a restricted product over $T$ of local factors, which are obtained by localising and completing along each flag. Often, one takes $T$ to be the set of all reduced chains on $X$, and we denote the resulting group by $\textbf{A}(X,M)$. $\textbf{A}(X,M)$ has more structure than that of an abelian group - it admits a semi-cosimplicial structure whose cohomology is that of $M$.
\newline Let $y$ be an irreducible curve on an arithmetic $\mathcal{S}$ as specified in section 3. If $T$ is the set of all reduced chains formed by closed points on $y$, then we will denote $\textbf{A}(\mathcal{S},\mathcal{O}_{\mathcal{S}},T)$ by $\textbf{A}(y)$. Later (remark~\ref{cannotextend.remark}) we will see that this space is ``too big'' for integration, which motivates us to introduce the smaller spaces of ``analytic'' adeles $\mathbb{A}(y)$, following the constructions of \cite[Section~1]{F3}\footnote{The geometric adeles $\textbf{A}(\mathcal{S})$ exist for arbitrary Noetherian schemes $\mathcal{S}$. One can consider what the general definition of the analytic space $\mathbb{A}(\mathcal{S})$ is and what role it plays in algebraic geometry.}.
\newline As mentioned in \ref{good.section}, there are two types of irreducible curves on $\mathcal{S}$ - vertical and horizontal curves. Whilst there is no real difference in the construction of $\mathbb{A}(y)$, we will treat the two cases separately so as to emphasize some important aspects in each setting. In particular, the fibres may well be singular, and the horizontal curves contain archimedean information.
\subsection{Fibres}\label{Fibres.subsubsection}
Let $y$ be an irreducible component of the fibre $\mathcal{S}_{\mathfrak{p}}$ over $\mathfrak{p}\in\text{Spec}(\mathcal{O}_k)$. For any $n\geq0$ and any point $x\in y$, one can define local lifting maps
\[l_{x,y}^n:E_{x,y}^{\oplus n}\rightarrow\begin{cases}
\mathcal{O}_{x,y}, & \text{if $K_{x,y}$ is of equal characteristic}.\\
\mathcal{O}_{x,y}/t^n\mathcal{O}_{x,y}, & \text{otherwise},
\end{cases}\]
and, subsequently, adelic lifting maps
\[L_{y}^n:\mathbb{A}(k(y))^{\oplus n}\rightarrow\begin{cases}
(K_{x,y})_{x\in y} \\
(\mathcal{O}_{x,y}/t^n\mathcal{O}_{x,y})_{x\in y}.
 \end{cases}\]
For details of these constructions, the reader is referred to \cite[Section~1.1]{F3}. The $y$-component of the analytic adeles is the following ring:
\[\mathbb{A}(y)=\{(a_{x,y})_{x\in y}:a_{x,y}\in K_{x,y},\forall n\geq0,~(a_{x,y})+t_y^n\mathcal{O}_y\in\text{im}(L_y^n)\}.\]
Recall that, when $x$ is a singular point on $y$, $K_{x,y}$ is in fact $\prod_{z\in y(x)}K_{x,z}$
\newline For $a_{x,y}\in\mathcal{O}_{x,y}=\prod_{z\in y(x)}\mathcal{O}_{x,z}$, let $\overline{a}_{x,y}=(\overline{a}_{x,z})_{z\in y(x)}$ denote the image of $a_{x,y}$ under the residue map to $\prod E_{x,z}$. We thus have
\[p_y:\mathbb{A}(y)\rightarrow\mathbb{A}(k(y))\]
\[(a_{x,y})\mapsto(\overline{a}_{x,y}).\]
\begin{definition}\label{ranktwoadeles.definition}
Let $\mathcal{S}_{\mathfrak{p}}$ denote the fibre of $\mathcal{S}$ over $\mathfrak{p}$. The $\mathcal{S}_{\mathfrak{p}}$-component $\mathbb{A}(\mathcal{S}_{\mathfrak{p}})$ of the analytic adeles is
\[\mathbb{A}(\mathcal{S}_{\mathfrak{p}})=\prod_{y\in\text{comp}(\mathcal{S}_{\mathfrak{p}})}\mathbb{A}(y).\]
We have a residue map
\[p_{\mathfrak{p}}=(p_y):\mathbb{A}(\mathcal{S}_{\mathfrak{p}})\rightarrow\prod_{y\in\text{comp}(\mathcal{S}_{\mathfrak{p}})}\mathbb{A}(k(y)).\]
\end{definition}
The following example gives a concrete interpretation of analytic adelic spaces and the subsequent remark explains why their measure theory cannot be extended to geometric adeles.
\begin{example}\label{adelictheory.example}
Let $S$ be a two-dimensional algebraic variety over a finite field and let $y$ be a nonsingular irreducible curve on $S$, with function field $k(y)$. Associated to $y$ we have the complete discrete valuation field
\[K_y=\text{Frac}(\widehat{\mathcal{O}_y}).\]
We fix a local parameter and denote it by $t_y$, it can be taken as a second local parameter for all two-dimensional local fields associated to closed points $x$ on $y$. We will refer to it as a local parameter of $y$. The ring $\mathbb{A}(k(y))$ is locally compact and has a Haar measure $\mu_{\mathbb{A}(k(y))}$.  We have a non-canonical isomorphism
\[\mathbb{A}(y)\cong\mathbb{A}(k(y))((t_y)).\]
Let $p$ be the map to $\mathbb{A}(k(y))$ sending a power series to its free coefficient. As in example~\ref{localtheory.example}, one can construct an $\mathbb{R}((X))-$measure $\mu_{y}$ on $\mathbb{A}(y)$. Let $S$ be a measurable subset of $\mathbb{A}(k(y))$, then
\[\mu_y(t_y^ip^{-1}(S))=X^i\mu_{\mathbb{A}(k(y))}(S).\]
\end{example}
\begin{remark}\label{cannotextend.remark}
As in the above situation, let $y$ be an irreducible curve on $\mathcal{S}$. If $\textbf{A}(y)$ is the group of geometric adeles associated to $y$ on $\mathcal{S}$, $M=\mathcal{O}_{\mathcal{S}}$ and $T$ is the set of all reduced chains of the form $x\in y\subset S$, then
\[\textbf{A}(y)=\cup_{r\in\mathbb{Z}}t_y^r\mathbb{A}(y).\]
$\textbf{A}(y)$ can be understood as a restricted direct product of $\textbf{A}(y)$ in which almost all components lie in $\mathbb{A}(y)$. Since the measure of $\mathbb{A}(k(y))$ is infinite, the measure of $\mathbb{A}(y)$ in the previous example is infinite. The geometric adeles are therefore a restricted product with respect to a set of infinite measure, and so we cannot extend the measure to $\textbf{A}$.
\end{remark}
\subsection{Horizontal Curves}
Horizontal curves on $\mathcal{S}$ will play a crucial role in this paper. We will begin by explaining their archimedean content, which is roughly that each horizontal curve intersects the archimedean fibres of the surface, as we now explain in more detail.
\newline By an archimedean fibre, we mean the fibre product
\[\mathcal{S}_{\sigma}=\mathcal{S}\times_{\text{Spec}(\mathcal{O}_k)}k_{\sigma}.\]
where $\sigma$ is an archimedean place of the base field $k$, with corresponding completion $k_{\sigma}$. There is a natural morphism from an archimedean fibre to the generic fibre
\[\mathcal{S}_{\sigma}\rightarrow\mathcal{S}\times_{\mathcal{O}_k}k=\mathcal{S}_{\eta}\cong C.\]
The fibre over any closed point on the generic fibre $C\cong\mathcal{S}_{\eta}$ is a finite reduced scheme. A horizontal curve $y$ on $\mathcal{S}$ is the closure $\overline{\{z\}}$ of a unique closed point $z\in C$, which has residue field $k(z)$. There are only finitely many points on $\mathcal{S}_{\sigma}$ which map to $z$, and they are the primes of $k_{\sigma}\otimes_kk(z)$, which correspond to the infinite places of $k(z)$ extending $\sigma$ on $k$.
\newline At a closed point $\omega$ on $\mathcal{S}_{\sigma}$ we have a two-dimensional local field \[K_{\omega,\sigma}=\text{Frac}(\widehat{\mathcal{O}_{\mathcal{S}_{\sigma},\omega}}).\]
The residue field of $K_{\omega,\sigma}$ is denoted $k_{\sigma}(\omega)$ and is either $\mathbb{R}$ or $\mathbb{C}$. We have, respectively
\[K_{\omega,\sigma}\cong\begin{cases}
\mathbb{R}((t)) \\
\mathbb{C}((t))
\end{cases}\]
Let $y$ be a horizontal curve on $\mathcal{S}$, and let $\sigma$ be an archimedean place of $k$. By the correspondence just described, we have an archimedean place $\omega$ of $k(y)$ and a two-dimensional local field
\[K_{\omega,y}=k(y)_{\omega}((t_y)),\]
where $k(y)_{\omega}$ is the completion of $k(y)$ at $\omega$.
\newline Repeating the construction from \ref{Fibres.subsubsection}, we obtain a lifting map
\[l_{\omega,y}^n:k(y)_{\omega}^{\oplus n}\rightarrow\mathcal{O}_{\omega,y}\cong\begin{cases}\mathbb{R}[[t]], \\ \mathbb{C}[[t]].\end{cases}\]
Also, at a closed point $x\in y$, we have a local lifting
\[l_{x,y}^n:E_{x,y}\rightarrow\mathcal{O}_{x,y}.\]
Altogether, we have an adelic map:
\[L_{y}^r:\oplus_{x\in y}l^n_{x,y}\oplus_{\omega}l_{\omega,y}^n:\mathbb{A}(k(y))^{\oplus n}\rightarrow\prod_{x\in y}K_{x,y}\prod_{\omega}K_{\omega,y}.\]
\begin{definition}\label{analytichorizontal.definition}
Let $y$ be a horizontal curve on $\mathcal{S}$. The $y$-component of the analytic adelic space is:
\[\mathbb{A}(y)=\{((a_{x,y})_{x\in y},(a_{\omega,y})_{\omega})\in\prod_{x\in y}K_{x,y}\prod_{\omega}K_{\omega,y}:\forall n\geq1,~((a_{x,y})_{x\in y},(a_{\omega,y})_{\omega})\in\text{im}(L_y^n)\}.\]
The residue maps $\mathcal{O}_{x,z}\rightarrow E_{x,z}$ and $\mathcal{O}_{\omega,y}\rightarrow k(y)_{\omega}$ induce
\[p_y:\mathbb{A}(y)\rightarrow\mathbb{A}(k(y)).\]
\end{definition}
\subsection{Additive Normalization}
We want to extend example~\ref{adelictheory.example} to the analytic adelic rings associated to fibres and horizontal curves. If $y$ is a curve on $\mathcal{S}$ of either description, the additive group of the ring $\mathbb{A}(k(y))$ of adeles on the function field $k(y)$ is a locally compact abelian group. It thus has a Haar measure, which is unique up to scalar multiplication. The $\mathbb{R}((X))$-measure on $\mathbb{A}(y)$ will depend on a choice of normalization of the Haar measure on $\mathbb{A}(k(y))$.
\newline Let $F$ be a two-dimensional local field. If $F$ is non-archimedean and $\psi_F:F\rightarrow\mathbb{C}^{\times}$ is a character, then we will refer to the orthogonal complement of $O_F$ as the conductor of $\psi_F$. If $F$ is an archimedean two-dimensional local field, the conductor is the orthogonal complement of $\mathcal{O}_F$.
\newline When $F=K_{x,z}$ (resp. $K_{\omega,y}$), we denote $\psi_F$ by $\psi_{x,z}$ (resp. $\psi_{\omega,y}$). The aim is to define the normalization of the measure on $\mathbb{A}(k(y))$ through the characters $\psi_{x,z}$.
\begin{lemma}\label{duality.lemma}
For any closed point $x$ on the fibre $\mathcal{S}_{\mathfrak{p}}$ of $\mathcal{S}$, let $z$ be a branch of an irreducible component of $\mathcal{S}_{\mathfrak{p}}$ at $x$. There are characters $\psi_{x,z}$ of the two-dimensional local fields $K_{x,z}$ such that if
\[\psi_{x,\mathfrak{p}}=\otimes_{z\in\mathcal{S}_{\mathfrak{p}}(x)}\psi_{x,z},\]
then the following is defined on $\mathbb{A}(\mathcal{S}_{\mathfrak{p}})$:
\[\psi_{\mathfrak{p}}=\otimes_{x\in\mathcal{S}_{\mathfrak{p}}}\psi_{x,\mathfrak{p}}\]
Moreover, the conductor $A_{x,\mathfrak{p}}$ is commensurable with $O_{x,\mathfrak{p}}$, with equality at almost all $x\in\mathcal{S}_{\mathfrak{p}}$, including the singular points. There is a non-trivial
\[\varphi_{\mathfrak{p}}:\mathbb{A}(k(\mathcal{S}_{\mathfrak{p}}))\rightarrow\mathbb{C}^{\times}\]
such that
\[\psi_{\mathfrak{p}}=\varphi_{\mathfrak{p}}p_{\mathfrak{p}}.\]
Similarly, if $y$ is a horizontal curve, for all points $x\in y$ and archimedean places $\omega$ of $k(y)$, there are local characters $\psi_{x,y}$ and $\psi_{\omega,y}$ such that
\[\psi_y=\otimes_{x\in y}\psi_{x,y}\otimes_{\omega}\psi_{\omega,y}\]
is defined on $\mathbb{A}(y)$ with the same properties.
\end{lemma}
\begin{proof}
\cite[Proposition~27]{F3}.
\end{proof}
From now on we fix such $\psi_{x,\mathfrak{p}}$ (resp. $\psi_{x,y}$) for all closed points on fibres $\mathcal{S}_{\mathfrak{p}}$ (resp. horizontal curves $y$).
\begin{definition}\label{dxz.definition}
Let $y$ be an irreducible component of a fibre. If $x$ is a nonsingular point on $y$, define $d(x,y)$ by
\[A_{x,y}=t^{d(x,y)}_{1,x,y}O_{x,y},\]
where $A_{x,y}$ is the conductor of $\psi_{x,y}$ and $t_{1,x,y}$ denotes the local parameter of $K_{x,y}$. If $x$ is a split ordinary double point on $y$, and $z,z'$ are local branches of $y$ at $x$, we will write $d(x,z)=d(x,z')=-1$.
\end{definition}
\begin{lemma}\label{classicalformula.lemma}
Let $\mathcal{S}_{\mathfrak{p}}$ be a smooth fibre on $\mathcal{S}$, then
\[\prod_{x\in\mathcal{S}_{\mathfrak{p}}}q_{x,\mathfrak{p}}^{d(x,\mathfrak{p})}=1\]
\end{lemma}
\begin{proof}
Let $k(\mathcal{S}_{\mathfrak{p}})$ denote the function field of $\mathcal{S}_{\mathfrak{p}}$ and let $v$ be a place of this global field. The residue field at $v$ has cardinality $q_v$. The lemma then follows from the representation of the canonical divisor $\mathfrak{C}$ on $z$ as
\[\mathfrak{C}=\sum_vm_vv,\]
where $P_v^{m_v}$ is the $v$-component of the conductor of the standard character on $\mathbb{A}(k(z))$. We know:
\[N(P_v^{m_v})=q_v^{m_v}=q^{\text{deg}(v)m_v}\]
So
\[\prod_vN(P_v^{m_v})=q^{-\text{deg}(\mathfrak{C})},\]
and formula follows from the fact that $\deg(\mathfrak{C})=2g-2$.
\end{proof}
Let $\mu_{\mathbb{A}(k(y))}$ be the Haar measure on $\mathbb{A}(k(y))$ which is self dual with respect to the character $\varphi_y$ from lemma~\ref{duality.lemma}.
If $y$ is an irreducible curve on $\mathcal{S}$, let $L$ denote a measurable subset of $\mathbb{A}(k(y))$ with respect to the above Haar measure. Consider the lifted measure $M_{\mathbb{A}(y)}$ on $\mathbb{A}(y)$ such that
\[M_{\mathbb{A}(y)}(t_{y}^ip_{y}^{-1}(L))=X^i\mu_{\mathbb{A}(k(y))}(L).\]
For example, let $\mathcal{S}_{\mathfrak{p}}$ be a smooth fibre of $\mathcal{S}$ over $\mathfrak{p}$. Consider the subset
\[O\mathbb{A}(\mathcal{S}_{\mathfrak{p}})=\prod_{x\in\mathcal{S}_{\mathfrak{p}}}O_{x,\mathfrak{p}},\]
then
\begin{align}
M_{\mathbb{A}(\mathcal{S}_{\mathfrak{p}})}(O\mathbb{A}(\mathcal{S}_{\mathfrak{p}}))&=\mu_{\mathbb{A}(k(\mathcal{S}_{\mathfrak{p}}))}(\prod_{x\in\mathcal{S}_{\mathfrak{p}}}O_x),\nonumber\\
&=\prod_{x\in\mathcal{S}_{\mathfrak{p}}}\mu_{k(\mathcal{S}_{\mathfrak{p}})_x}(O_x)\nonumber\\
&=\prod_{x\in\mathcal{S}_{\mathfrak{p}}}q_{x,\mathfrak{p}}^{d(x,y)/2}\nonumber\\
&=q_{\mathfrak{p}}^{1-g(\mathcal{S}_{\mathfrak{p}})},\nonumber
\end{align}
where $g(\mathcal{S}_{\mathfrak{p}})=g$ denotes the genus of the special fibre $\mathcal{S}_{\mathfrak{p}}$.
\begin{definition}\label{renormalizedcurvemeasure.definition}
Let $g$ be the genus of $C$ and $\mathcal{S}_{\mathfrak{p}}$ be a smooth fibre, define:
\[\mu_{\mathbb{A}(\mathcal{S}_{\mathfrak{p}})}=q(\mathcal{S}_{\mathfrak{p}})^{g-1}M_{\mathbb{A}(y)}.\]
If $\mathcal{S}_{\mathfrak{p}}$ is a singular fibre, then
\[\mu_{\mathbb{A}(\mathcal{S}_{\mathfrak{p}})}=M_{\mathbb{A}(\mathcal{S}_{\mathfrak{p}})}.\]
If $y$ is a horizontal curve, let $\mu_{\mathbb{A}(y)}=M_{\mathbb{A}(y)}$.
\end{definition}
So, for all smooth fibres and horizontal curves $\mu_{\mathbb{A}(y)}(O\mathbb{A}(y))=1$.
\begin{definition}
A simply integrable function on $\mathbb{A}(y)$ is a finite linear combination of characteristic functions of measurable sets under the measure $\mu_{\mathbb{A}(y)}$.
\end{definition}
In this paper we will only integrate these simply integrable functions, for which the integral is the linear combination of the measures. For a more general theory, see \cite[1.3]{F3}.
\section{Zeta Integrals on $\mathcal{S}$}\label{zetaintegral.section}
Two-dimensional zeta integrals were first studied by Fesenko for proper regular models of elliptic curves \cite{F4,F3}. We extend his results to a model $\mathcal{S}$ as in \ref{good.section}, following the sketch in \cite[Part~57]{F4}.
\subsection{Multiplicative Normalization}
First, we recall the relationship between the measure on the additive and multiplicative group of one-dimensional local fields and adeles.
\begin{example}\label{addvsmultdim1.example}
Let $k$ be a number field. At each non-archimedean prime $\mathfrak{p}$ we have a normalized Haar measure $d\mu$ on the locally compact additive abelian group $k_{\mathfrak{p}}$, which has finite residue field of cardinality $q(\mathfrak{p})$. One then integrates on the multiplicative group $k_{\mathfrak{p}}^{\times}$ with the measure
\[(1-q(\mathfrak{p})^{-1})\frac{d\mu}{~~|~|_{k_{\mathfrak{p}}}}.\]
In turn, one integrates over the idele group $\mathbb{A}_k^{\times}$ with the tensor product of these measures.
\end{example}
Similarly, on an arithmetic surface, we need a measure compatible with the multiplicative structure of a two-dimensional local field. Let $F$ be such a field with local parameter $t_2$ and let $t_1$ be a lift of the local parameter of the residue field of $F$. Let $U$ denote the group of principal units, and let $q$ be the cardinality of the final (finite) residue field. One can decompose the multiplicative group $F^{\times}$ as follows:
\[F^{\times}=<t_1>^{\times}<t_2>^{\times}U.\]
Using this decomposition we define the $\mathbb{R}((X))-$valued module $|~|_F$ by
\[|t_2^it_1^ju|_F=q^{-j}X^i.\]
When $F=K_{x,y}$, we use the notation $|~|_F=|~|_{x,y}$.
\newline Let $z$ denote a local branch of an irreducible curve on $\mathcal{S}$ at a point $x$. Motivated by example~\ref{addvsmultdim1.example}, we will use the following measure on $K_{x,z}^{\times}$:
\[M_{K_{x,z}^{\times}}=\frac{M_{K_{x,z}}}{(1-q(x,z)^{-1})|~|_{x,z}}.\]
\begin{definition}
A simply integrable function with respect to the above measure is a finite linear combination of characteristic functions of measurable sets. 
\end{definition}
\begin{example}\label{nonsingular.example}
Let $\mathcal{S}_{\mathfrak{p}}$ be a smooth fibre of $\mathcal{S}$ and consider the following measurable function for each closed $x\in y$
\[f_{x,\mathcal{S}_{\mathfrak{p}}}=|~|^s_{x,\mathcal{S}_{\mathfrak{p}}}\text{char}(O_{x,\mathfrak{p}}),\]
and define $f_{\mathcal{S}_{\mathfrak{p}}}=\otimes_{x\in\mathcal{S}_{\mathfrak{p}}}f_{x,\mathcal{S}_{\mathfrak{p}}}$. Then, $f_{\mathcal{S}_{\mathfrak{p}}}$ is integrable and
\[\int_{\mathbb{A}(\mathcal{S}_{\mathfrak{p}}^{\times})}f_{\mathcal{S}_{\mathfrak{p}}}d\mu_{\mathbb{A}(\mathcal{S}_{\mathfrak{p}})}^{\times}=\zeta(\mathcal{S}_{\mathfrak{p}},s)\prod_{x\in\mathcal{S}_{\mathfrak{p}}}q(x,\mathcal{S}_{\mathfrak{p}})^{(d(x,\mathcal{S}_{\mathfrak{p}}))(1-s)}=\zeta(\mathcal{S}_{\mathfrak{p}},s)\prod_{x\in\mathcal{S}_{\mathfrak{p}}}q(x,\mathcal{S}_{\mathfrak{p}})^{(1-g)(1-s)}.\]
\end{example}
In our case the special fibres have at worst split ordinary double singularities and we will use an ad hoc variant of the function in example~\ref{nonsingular.example} to recover the corresponding factor of the zeta-function - for a more complete approach see \cite[36,~Remark~1,~37]{F3}. When $x$ is a singular point of $\mathcal{S}_{\mathfrak{p}}$, define
\[M_{K_{x,\mathcal{S}_{\mathfrak{p}}}^{\times}}=\otimes_{z\in \mathcal{S}_{\mathfrak{p}}(x)}M_{K^{\times}_{x,z}}.\]
\subsection{Zeta Integrals on the Projective Line}
We would like to take the product over all the fibres in order to obtain the non-archimedean part of the zeta function of $\mathcal{S}$, including the conductor. Unfortunately, the product diverges due to the additional factors appearing in example~\ref{nonsingular.example}.
\newline To resolve this, we begin by observing something complementary that happens when we apply the adelic analysis on the scheme $\mathcal{P}:=\mathbb{P}^1(\mathcal{O}_k)$. At a non-archimedean place $\mathfrak{p}$ of the base field $k$, the fibre $\mathcal{P}_{\mathfrak{p}}=\mathbb{P}^1(k(\mathfrak{p}))$. At a closed point $x\in\mathcal{P}_{\mathfrak{p}}$, define
\[g_{x,\mathfrak{p}}=\text{char}(O_{x,\mathcal{P}_{\mathfrak{p}}}),\]
and subsequently,
\[g_{\mathfrak{p}}=\otimes_{x\in\mathcal{P}_{\mathfrak{p}}}g_{x,\mathfrak{p}}.\]
Then
\[\int_{\mathbb{A}(\mathcal{P}_{\mathfrak{p}})^{\times}}g_{\mathfrak{p}}|~|^s_{\mathcal{P}_{\mathfrak{p}}}d\mu_{\mathbb{A}(\mathcal{P}_\mathfrak{p})^{\times}}=\zeta(\mathcal{P}_{\mathfrak{p}},s)\prod_{x\in\mathcal{P}_{\mathfrak{p}}} q_x^{(1-s)}.\]
Combining this computation with examples~\ref{nonsingular.example} and~\ref{singular.example} (later), we see that if $\mathfrak{p}$ is a good prime of $C$, then
\[\int_{\mathbb{A}(\mathcal{S}_{\mathfrak{p}})^{\times}}f_{\mathfrak{p}}|~|^s_{\mathcal{S}_{\mathfrak{p}}}d\mu_{\mathbb{A}(\mathcal{S}_{\mathfrak{p}})^{\times}}\cdot\left(\int_{\mathbb{A}(\mathcal{P}_{\mathfrak{p}})^{\times}}g_{\mathfrak{p}}|~|^s_{\mathcal{P}_{\mathfrak{p}}}d\mu_{\mathbb{A}(\mathfrak{p})^{\times}}\right)^{g-1}.\]
\[=\zeta(\mathcal{P}_{\mathfrak{p}},s)^{1-g}\zeta(\mathcal{S}_{\mathfrak{p}},s).\]
This is essentially a non-archimedean factor of the zeta integral in~\ref{zetaintegral.subsubsection} below - complications will arise at the bad primes.
\newline The $(1-g)$th power of the zeta integral on $\mathbb{P}^1(\mathcal{O}_k)$ conveniently cancels the divergent part of the zeta integral over $\mathcal{S}$. But that is not all, as by a completion process for the zeta function of $\mathbb{P}^1(\mathcal{O}_k)^{1-g}$, we can recover the gamma factor of $\mathcal{S}$ up to an $[s\mapsto2-s]-$invariant rational function. We will now make this idea precise.
\subsection{The Gamma Factor}
The gamma factor for the zeta function of $\mathcal{S}$ is the quotient of the gamma factors of it's Hasse--Weil decomposition, i.e..
\[\Gamma(\mathcal{S},s)=\frac{\Gamma(k,s)\Gamma(k,s-1)}{\Gamma(C,s)}.\]
The renormalizing factor in fact induces the gamma factor in a very natural way. The zeta function of $\mathbb{P}^1(\mathcal{O}_k)$ is very simple:
\[\zeta(\mathbb{P}^1(\mathcal{O}_k),s)=\zeta(k,s)\zeta(k,s-1),\]
and so its gamma factor is
\[\Gamma(\mathbb{P}^1(\mathcal{O}_k),s)=\Gamma(\mathbb{C},s)^{r_2}\Gamma(\mathbb{R},s)^{r_1}\Gamma(\mathbb{C},s-1)^{r_2}\Gamma(\mathbb{R},s-1)^{r_1}.\]
Applying well-known identities of the gamma function,
\begin{align}
\Gamma(\mathbb{P}^1(\mathcal{O}_k),s)^{1-g}&=\frac{1}{(\Gamma(\mathbb{C},s)^{r_2}\Gamma(\mathbb{C},s-1)^{r_2}\Gamma(\mathbb{R},s)^{r_1}\Gamma(\mathbb{R},s-1)^{r_1})^{g-1}}\nonumber\\
&=\frac{1}{(\Gamma(\mathbb{C},s)^{r_2}\Gamma(\mathbb{C},s-1)^{r_2}\Gamma(\mathbb{C},s-1)^{r_1})^{g-1}}\nonumber\\
&=\frac{\pi^{-r_2(g-1)}(s-1)^{r_2(g-1)}}{(\Gamma(\mathbb{C},s)^{2r_2}\Gamma(\mathbb{C},s-1)^{r_1})^{g-1}}\nonumber\\
&=\frac{\pi^{-(r_1+r_2)(g-1)}(s-1)^{(r_1+r_2)(g-1)}}{(\Gamma(\mathbb{C},s)^{r_1+2r_2})^{g-1}}\nonumber\\
&=\frac{\pi^{-(r_1+r_2)(g-1)}(s-1)^{(r_1+r_2)(g-1)}}{R(s)}\Gamma(\mathcal{S},s)\nonumber\\
&=Q(s)\Gamma(\mathcal{S},s)\nonumber,
\end{align}
where
\[Q(s)=\frac{\pi^{-(r_1+r_2)(g-1)}(s-1)^{(r_1+r_2)(g-1)}}{R(s)},\]
so
\[Q(2-s)=\pm Q(s).\]
We thus see that completing the normalizing factor gives us the transcendental part of $\Gamma(\mathcal{S},s)$.
\subsection{Integration on Horizontal Curves}
We will remind ourselves of the Haar measure on $\mathbb{R}$ and $\mathbb{C}$.
\[\mu_{k_{\sigma}(\omega)}=\begin{cases}
\text{Lebesgue measure, }dx & k_{\sigma}(\omega)=\mathbb{R}, \\
\text{twice Lebesgue measure, }2dz & k_{\sigma}(\omega)=\mathbb{C}.
\end{cases}\]
One then integrates on the multiplicative group $\mathbb{R}^{\times}$ (resp. $\mathbb{C}^{\times}$) with the measure $\frac{dx}{|x|}$ (resp. $\frac{2dx}{|z|^2}$).
\begin{example}\label{archimedean.example}
We have the well-known identities:
\[\int_{\mathbb{R}^{\times}}e^{-\pi x^2}|x|^s\frac{dx}{x}=\Gamma(\mathbb{R},s),\]
\[\int_{\mathbb{C}^{\times}}e^{-2\pi |z|^2}|z|^s\frac{2dz}{|z|^2}=2\pi\Gamma(\mathbb{C},s).\]
These are precisely the Gamma factors required for the Dedekind-zeta function of a number field at a real (respectively complex) place.
\end{example}
We will integrate on $K_{\omega,y}$ with the lifted measure from $k(y)_{\omega}$. 
\begin{definition}
A basic measurable set on $K_{\omega,y}$ is one of the form $t_{\omega}^i\mu_{k(y)_{\omega}}(A)$, where $A\subset k(y)_{\omega}$ is measurable with respect to the Haar measure on the archimedean local field $k(y)_{\omega}$. This can be extended to finite unions by linearity.
\end{definition}
At all closed points $x$ of $y$ we have the two-dimensional local field $K_{x,y}$ and the natural lifted Haar measure as described in 6.1. Altogether, we have a measure on $\mathbb{A}(y)=\prod_x'K_{x,y}\prod_{\omega}\mathbb{A}(\omega,y)$ for a horizontal curve $y$. For reasons that will soon be apparent, we redefine\footnote{This potentially confusing notation will be used throughout without much further comment} $|~|_y$ to be $\prod_{x\in y}|~|_y^{1/2}$. 
\newline On a set $S$ of fibres and finitely many horizontal curves:
\[|~|_{S}=\prod_{y\in S}|~|_y,\]
where
\[|~|_y=\prod_{x\in y}|~|_{x,y}.\]
On a horizontal curve $y$, we redefine $\mathbb{A}(y)^{\times}$ to be a maximal subgroup such that the image of $|~|_y$ is equal to that of $|~|_y^2$.
\begin{example}\label{horizontal.example}
At a nonsingular $x\in y$, let $f_{x,y}=\text{char}(O_x)$. At an archimedean place $\omega$ of $y$, let $f_{\omega,y}(\alpha)=\text{char}(\mathcal{O}_{\omega,y})(\alpha)\text{exp}(\text{Tr}_{k(y)_{\omega}/\mathbb{R}}(1)|\text{res}_{t^0_{\omega}}(\alpha)|)$. Then
\[\int_{\mathbb{A}(y)^{\times}}f|~|_y^sd\mu_{\mathbb{A}(y)^{\times}}=\zeta(k,\tilde{f},|~|_k^{s/2}),\]
where $\zeta(k,g,\chi)$ is a classical Iwasawa--Tate zeta integral and
\[\tilde{f}=\otimes_v\tilde{f}_v\]
\[\tilde{f}_v(\alpha_v)=\begin{cases}
\text{char}(\mathcal{O}_v)(\alpha_v), & v\text{ archimedean} \\
e^{-\pi\alpha_v^2}, & v\text{ real} \\
e^{-2\pi|\alpha_v|^2}, & v\text{ complex}
\end{cases}\]
by the well-known theory of Iwasawa--Tate this integral defines a meromorphic function on $\mathbb{C}$ and satisfies a functional equation with respect to $s\mapsto2-s$.
\end{example}
\subsection{Zeta Integrals}\label{zetaintegral.subsubsection}
We are missing the factors at bad primes $\mathfrak{p}$. At a split ordinary double singularity on the fibre $\mathcal{S}_{\mathfrak{p}}$ we have two local branches, so that integrating over multiplicative group of the two-dimensional analytic adelic space for $\mathcal{S}_{\mathfrak{p}}$ gives us an additional factor that is not present in the zeta function. One way of treating singular and smooth fibres $y$ in a regular way is by integrating over $\mathbb{A}(y)^{\times}\times\mathbb{A}(y)^{\times}$, which we give the product measure.
\begin{example}\label{singular.example}
Let $y=\mathcal{S}_{\mathfrak{p}}$ be a fibre over $\mathfrak{p}$ with singular point $x$. Let $z$ be a branch of $y$ at $x$, then define $f_{x,y}$ on $O_{x,y}\times O_{x,y}$ as follows\footnote{This is the image of $\text{Char}(O_{x,y})$ under Fesenko's ``diamond operator''.}:
\[f_{x,y}=q_x^{-1}\text{char}(O_{x,z},t_{1,x,z}^{-1}O_{x,z}).\]
For nonsingular points $x\in y$ put
\[f_{x,y}=\text{char}(O_{x,y},O_{x,y}).\]
Combining, put $f_y=\otimes_{x\in y}f_{x,y}$, then
\[\int f_yd\mu_{\mathbb{A}(y)^{\times}\times\mathbb{A}(y)^{\times}}=A_{\mathfrak{p}}(\mathcal{S})^{(1-s)}\zeta(y,s)^2\prod_{z\in\mathcal{S}_{\mathfrak{p}}}q_z^{2(1-g_z)(1-s)}.\]
\end{example}
All that remains is to put everything together as an integral over the whole adelic space $\mathbb{A}(S)^{\times}$, where $S$ is contains all fibres of $\mathcal{S}$ and finitely many horizontal curves.
\begin{definition}\label{cnf.definition}
Combining the previous examples, let
\[f=\otimes_{y\in S}f_y,\]
where $y$ runs over all curves in $S$. $f_y$ is defined as follows:
\begin{enumerate}
\item Let $y$ be a nonsingular fibre and $x\in y$ be a closed point, put
\[f_{x,y}=\text{char}((O_{x,y},O_{x,y})),\]
\[f_y=\otimes f_{x,y}\]
\item Let $y$ be a fibre with singular point $x$. Choose branches $z,z'\in y(x)$ and put \[f_{x,y}=q_x^{-1}\text{char}(O_{x,z},t_{1,x,z}O_{x,z}).\]
\item Let $y$ be a (nonsingular) horizontal curve, for non-archimedean places of $k$ define $f_{x,y}$ as in point (1). At archimedean places $\omega$ take
    \[f_{\omega,y}(\alpha)=\text{char}(\mathcal{O}_{\omega,y})(\alpha)\text{exp}(\text{Tr}_{k(y)_{\omega}/\mathbb{R}}(1)|\text{res}_{t^0_{\omega}}(\alpha)|).\]
\end{enumerate}
\end{definition}
We introduce the following abbreviated notation:
\begin{definition}
Let $k$ be a number field and $\mathcal{P}=\mathbb{P}^1(\mathcal{O}_k)$. Let $S$ denote a set of curves on $\mathcal{S}$, consisting of all fibres and a finite set $S_{h}$ of horizontal curves on $\mathcal{S}$. If $f$ is an integrable function on $\mathbb{A}(\mathcal{S})^{\times}\times\mathbb{A}(\mathcal{S})^{\times}$ and $h$ is an integrable function on $\mathbb{A}(\mathcal{P})^{\times}\times\mathbb{A}(\mathcal{P})^{\times}$, then the zeta integral $\zeta^{(2)}_S(f,h,s)$ is defined to be the following product:
\[\prod_{\mathfrak{p}\in\text{Spec}\bigl(\mathcal{O}_k\bigr)}(\int_{\mathbb{A}(\mathcal{P}_{\mathfrak{p}})^{\times}\times\mathbb{A}(\mathcal{P}_{\mathfrak{p}})^{\times}}h_{\mathfrak{p}}|~|_{\mathcal{P}_{\mathfrak{p}}}^sd\mu_{\mathbb{A}(\mathcal{P}_{\mathfrak{p}})^{\times}})^{g-1}\int_{\mathbb{A}(\mathcal{S}_{\mathfrak{p}})^{\times}\times\mathbb{A}(\mathcal{S}_{\mathfrak{p}})^{\times}}f_{\mathcal{S}_{\mathfrak{p}}}|~|^s_{\mathcal{S}_{\mathfrak{p}}}d\mu_{\mathbb{A}(\mathcal{S}_{\mathcal{P}})^{\times}}\]
\[\times\prod_{y\in S_{h}}\int_{\mathbb{A}(y)^{\times}}f_y|~|^s_yd\mu_{\mathbb{A}(y)^{\times}}\times\bigl(\xi(\mathbb{P}^1(\mathcal{O}_k),s)\bigr)^{1-g}.\]
\end{definition}
\begin{remark}
This can be viewed as a ``renormalized'' integral over the adelic spaces $\mathbb{A}(\mathcal{S},S)^{\times}$ and $\mathbb{A}(\mathcal{P},S)^{\times}$ \cite[Part~57]{F3}. In the next section we will consider this as an integral over the analytic adeles of the non-connected arithmetic scheme
\[\mathcal{S}\coprod_{i=1}^{g-1}\mathcal{P}\]
\end{remark}
\begin{remark}
In this section we have only specified one integrable function, for a more complete theory see \cite[Section~1.3]{F3}. In general, integrable functions will only differ at finitely many components.
\end{remark}
At a closed point $x\in\mathcal{P}_{\mathfrak{p}}$, define $h_{x,\mathcal{P}_{\mathfrak{p}}}=\text{char}(O_{x,\mathcal{P}_{\mathfrak{p}}},O_{x,\mathcal{P}_{\mathfrak{p}}})$, and let $h_{\mathcal{P}_{\mathfrak{p}}}=\otimes_{x\in\mathcal{P}_{\mathfrak{p}}}h_{x,\mathcal{P}_{\mathfrak{p}}}$. Convergence of the preliminary zeta integral in some specified half plane will be a corollary (\ref{Convergence.corollary}) of the following computation.
\begin{theorem}\label{zetawithhorizontal.theorem}
Let $S$ be a set of curves consisting of all fibres and finitely many horizontal $y_i$, each of function field $k(y_i)$. If $f$ is as in definition~\ref{cnf.definition} and $h$ is as above, then
\[\zeta^{(2)}(f,h,s)=Q(s)^2\Gamma(\mathcal{S},s)^2A(\mathcal{S})^{(1-s)}\zeta(\mathcal{S},s)^2\prod_i\xi(k(y_i),s/2)^2,\]
where $Q(s)$ is a rational function such that
\[Q(s)=\pm Q(2-s),\]
and $\xi(k(y_i),s)$ is the completed Dedekind zeta function of the finite extension $k(y_i)/k$ which satisfies the functional equation:
\[\xi(k(y_i))\bigl(\frac{s}{2}\bigr)=\xi(k(y_i))\bigl(\frac{2-s}{2}\bigr).\]
\end{theorem}
\begin{proof}
This follows from combining examples~\ref{nonsingular.example},~\ref{singular.example},~\ref{archimedean.example} and~\ref{horizontal.example}.
\end{proof}
\begin{corollary}\label{Convergence.corollary}
Assume that the integral in the above definition is defined at $f,h$, then it converges for $s\in\{\Re(s)>2\}$.
\end{corollary}
\begin{proof}
In the case of $f,h$ as in the theorem, the convergence of the zeta integral follows from the well known properties of $\zeta(\mathcal{S},s)$ which are described in \cite{ZALF}. For arbitrary $f,g$ such that the zeta integral is defined, the calculation will differ by only finitely many factors.
\end{proof}
We will introduce the notation
\[\mathcal{Z}(\mathcal{S},\{y_i\},s)=\zeta(\mathcal{S},s)A(\mathcal{S})^{(1-s)/2}\Gamma(\mathcal{S},s)Q(s)\prod_i\xi(k(y_i),s/2).\]
Clearly, the zeta function verifies admits meromorphic continuation if and only if $\mathcal{Z}(\mathcal{S},\{y_i\},s)$ does, and the functional equations are equivalent.
\newline Let $\mathcal{T}$ be a two-dimensional arithmetic scheme over $\text{Spec}(\mathcal{O}_k)$. For $\mathfrak{p}\in\text{Spec}(\mathcal{O}_k)$, define $|~|^{(n)}_{\mathcal{T}_{\mathfrak{p}}}$ on $(\mathbb{A}(\mathcal{T}_{\mathfrak{p}})^{\times})^{\times n}$ by $|(a_1,\dots,a_n)|^{(n)}=|a_1|\dots|a_n|$. We will use the product measure on $(\mathbb{A}(\mathcal{T}_{\mathfrak{p}})^{\times})^{\times n}$. Let $f^{(n)}=(f,\dots,f)$ and $g^{(n)}=(g,\dots,g)$, and define $\zeta^{(n)}(\mathcal{S},f,h,s)$ as the following product:
\[\prod_{\mathfrak{p}\in\text{Spec}(\mathcal{O}_k)}\bigl(\int_{(\mathbb{A}(\mathcal{P}_{\mathfrak{p}})^{\times})^{\times n}}g^{(n)}_{\mathfrak{p}}(|~|_{\mathcal{P}_{\mathfrak{p}}}^{(n)})^sd\mu_{(\mathbb{A}(\mathcal{P}_{\mathfrak{p}})^{\times})^{\times n}}\bigr)^{g-1}\int_{(\mathbb{A}(\mathcal{S}_{\mathfrak{p}})^{\times})^{\times n}}f^{(n)}_{\mathcal{S}_{\mathfrak{p}}}(|~|^{(n)}_{\mathcal{S}_{\mathfrak{p}}})^{s}d\mu_{(\mathbb{A}(\mathcal{S}_{\mathcal{P}})^{\times})^{\times n}} \]
\[\times\prod_{y\in S_{h}}\int_{(\mathbb{A}(y)^{\times})^{\times n}}f^{(n)}_y(|~|^{(n)})^s_yd\mu_{(\mathbb{A}(y)^{\times})^{\times n}}\xi(\mathbb{P}^1(\mathcal{O}_k),s)^{n(1-g)/2}.\]
\begin{corollary}\label{firstcalculation.corollary}
For each positive integer $m$, we have
\[\zeta^{(2m)}(\mathcal{S},f,h,s)=\mathcal{Z}(\mathcal{S},\{y_i\},s)^{2m}.\]
\end{corollary}
\begin{proof}
This follows from Theorem~\ref{zetawithhorizontal.theorem} and the definitions of measures above.
\end{proof}
When the genus of $C$ is $1$ and $2m=2$, we recover Fesenko's zeta integrals for elliptic curves and the formula in corollary~\ref{firstcalculation.corollary} agrees with his first calculation. This motivates the following definitions:
\begin{definition}
An simply integrable function on $(\mathbb{A}(\mathcal{S},S)^{\times})^{\times 2}$ is a finite linear combination of $\mu_{(\mathbb{A}(\mathcal{S},S)^{\times})^{\times 2}}$-measurable sets.
\end{definition}
\begin{definition}
Let $S$ be a set of curves on $\mathcal{S}$, for simply integrable functions \[f:(\mathbb{A}(\mathcal{S},S)^{\times})^{\times 2}\rightarrow\mathbb{C}\]
and $h$ on $(\mathbb{A}(\mathbb{P}^1(\mathcal{O}_k))^{\times})^{\times 2}$, the ``two-dimensional unramified zeta integral'' is
\[\zeta(\mathcal{S},S,f,h,|~|^s):=\zeta^{(2)}(S,f,h,s).\]
\end{definition}
We make the following conjecture, extending that of \cite[Section~4]{F3}:
\begin{conjecture}\label{ACFEZI.conjecture}
Provided the set $S$ of curves on $\mathcal{S}$ contains finitely many horizontal curves, the zeta integral $\zeta(\mathcal{S},S,f,h,|~|^s)$ meromorphically extends to the complex plane and satisfies the following functional equation
\[\zeta(\mathcal{S},S,f,h,|~|^s)=\zeta(\mathcal{S},S,f,h,|~|^{2-s}).\]
\end{conjecture}
\begin{remark}
Let $S$ contain all fibres of $\mathcal{S}$ and finitely many horizontal curves, then
\begin{align*}
\zeta(\mathcal{S},S,f,h,|~|^s)=\zeta(\mathcal{S},S,f,h,|~|^{2-s})\Longleftrightarrow&\xi(\mathcal{S},s)^2=\xi(\mathcal{S},2-s)^2 \\
\Longleftrightarrow&\xi(\mathcal{S},s)=\pm\xi(\mathcal{S},2-s)\\
\Longleftrightarrow&\Lambda(C,s)=\pm\Lambda(C,2-s).
\end{align*}
\end{remark}
We have integrated over two copies of the multiplicative group of the ring of analytic adeles so as to get the correct factor of the zeta function at split ordinary double points. This is not the only motivation for doing so. In fact, there is a certain compatibility with two-dimensional class field theory that allows us to define ``twisted'' zeta integrals whose evaluation is an analogue of Hecke $L$-functions for arithmetic surfaces. This will be the subject of a later section. Before then, we will formulate the mean periodicity correspondence in terms of this ``two dimensional adelic analysis'' on $\mathcal{S}$.
\section{Adelic Duality and Filtrations}
In dimension two, there are three ``levels'' to the adeles - formal definitions will be given in 7.3. On the purely local level, one has the products of fields associated to closed points on irreducible curves. The other extreme is the global object, i.e. the function field of the surface. In between one has the local-global (or semi-global) complete discrete valuation fields associated to irreducible curves, or closed points on the surface. One may consider these levels as a filtration on the adeles, from which one constructs semi-cosimplicial complexes which compute the cohomology of quasi-coherent sheaves. The additive duality of the adeles and associated quotients can then be used to deduce the Riemann--Roch theorem, for example as in \cite{GAATRRFOCOS}. On the other hand, the duality of multiplicative, and $K$-theoretic, adelic structures give rise to results in class field theory \cite[Chapter~IX]{FV}. 
\newline It is our desire to apply additive and multiplicative adelic duality to the zeta integrals of the previous section. Shortly we will derive a harmono-analytic expression of adelic duality known as the ``two-dimensional theta formula.'' The analogous expression for elliptic surfaces was first proved by Fesenko in \cite[\S3.6]{F3}. This terminology is by analogy to the classical theta formula, expressing the functional equation of the theta function. This classical result can be verified through Poisson summation on the adeles and is used in the Iwasawa--Tate method for the functional equation of Hecke $L$-functions. 
\newline In order to achieve this, we must construct integrals on the local-global adelic spaces. The measures required do not factorize as a product of local factors, so the ad hoc method of renormalizing in the previous section is not sufficient. Instead, we will consider convergent integrals on certain non-connected arithmetic schemes.
\newline In 7.1 we will introduce the ``boundary functions'' associated to $\zeta(\mathcal{S},s)$, for arithmetic surfaces $\mathcal{S}$. They are defined in terms of inverse Mellin transforms of products of zeta functions and are a priori nothing to do with two-dimensional adeles. We will realise these functions as adelic integrals in 7.2. In 7.3, the two-dimensional theta formula will allow us to understand an important component of the boundary function as an adelic integral. Indeed, the terminology ``boundary'' comes from the fact that this component is an integral over a semi-global adelic boundary, with respect to a somewhat complicated topology. The meromorphic continuation, functional equation, and even poles, of the zeta integrals are all reduced to the analogous properties of this boundary integral.
\subsection{Boundary Functions and Mean-Periodicity}
The mean-periodicity correspondence provides a necessary and sufficient condition under which arithmetic zeta functions admit meromorphic continuation and the expected functional equation. Our aim in this subsection is to provide some intuition of this result. For simplicity we will work in the strong Schwartz space $\textbf{S}(\mathbb{R}^{\times}_{+})$ of functions $\mathbb{R}^{\times}_{+}$\footnote{Equally one could work with strong Schwartz functions on $\mathbb{R}$ by composing with the exponential map.}. This space is used in the theory of zeta functions elsewhere, for example \cite{OAROTICGRTPAZOLF}, and is even implicit as far back as \cite{FZED}. It's topological dual is the space of weak tempered distributions.
\begin{definition}
A strong Schwartz function is $\textbf{S}((\mathbb{R}^{\times}_{+}))$-mean-periodic if there is a non-trivial weak-tempered distribution $f^{\ast}$ such that
\[f\ast f^{\ast}=0.\]
Equivalently\footnote{This is equivalent as the Hahn-Banach theorem holds in $\textbf{S}((\mathbb{R}^{\times}_{+}))$.}, $f$ is $\textbf{S}((\mathbb{R}^{\times}_{+}))$-mean-periodic if
\[\overline{\text{Span}_{\mathbb{C}}\{y\cdot f:y\in\mathbb{R}^{\times}_{+}\}}\neq\textbf{S}((\mathbb{R}^{\times}_{+})),\]
where
\[y\cdot f(x)=f(x/y).\]
\end{definition} 
\begin{example}
One should think about this definition in comparison to smooth periodic functions on $\mathbb{R}$. Say such a function $f$ has period $p$, then it satisfies the following convolution equation
\[f\ast(\delta_p-\delta_0)=0,\]
where $\delta_a$ denotes the Dirac distribution at $a$. Evidently, $\mathbb{R}$ acts on $\mathscr{C}^{\infty}(\mathbb{R})$ by $y:f(x)\mapsto f(x-y)$, and the set Span$_{\mathbb{C}}\{y\cdot f:y\in\mathbb{R}\}$ is not dense in $\mathscr{C}^{\infty}(\mathbb{R})$. Unfortunately, it was explained in \cite{SRF} that the space of smooth functions is not suitable for applications to zeta functions, though, the space of functions of not more than exponential growth may be.
\end{example}
Let $\mathcal{S}\rightarrow\text{Spec}(\mathcal{O}_k)$ be a proper, regular model of a smooth projective curve over a number field $k$. For $i=0,\dots,n$, let $k_i$ denote a finite Galois extension of $k$. Define, for $c\gg0$, the inverse Mellin transform
\[f(\mathcal{S},\{k_i\}):\mathbb{R}^{\times}_{+}\rightarrow\mathbb{C}\]
\[f(\mathcal{S},\{k_i\},x)=\frac{1}{2\pi i}\int_{(c)}\mathcal{Z}(\mathcal{S},\{k_i\},s)x^{-s}ds\]
and
\[h(\mathcal{S},\{k_i\}):\mathbb{R}^{\times}_{+}\rightarrow\mathbb{C}\]
\[h(\mathcal{S},\{k_i\},x)=f(\mathcal{S},\{k_i\},x)-x^{-1}f(\mathcal{S},\{k_i\},x^{-1})\]
It follows from \cite[Theorem~5.18]{SRF} that $\zeta(\mathcal{S},s)$ admits meromorphic continuation to $\mathbb{C}$ and satisfies the functional equation (up to sign) if and only if there is an integer $n$ and field extensions $k_i/k$, $i=1,\dots,n$, such that $h(\mathcal{S},\{k_i\},x)$ is $\textbf{S}((\mathbb{R}^{\times}_{+}))$-mean-periodic. We are lead to the following conjecture.
\begin{conjecture}
There exists a finite set of extensions $\{k_i/k\}$ such that $h(\mathcal{S},\{k_i\},x)$ is $\mathfrak{X}$-mean-periodic, where $\mathcal{X}=\textbf{S}(\mathbb{R}^{\times}_{+})$-mean-periodic.
\end{conjecture}
The remainder of this section will focus on providing a two-dimensional adelic framework for studying this conjecture. More precisely, we will for obtain integral representations for the functions $h(\mathcal{S},\{k_i\},x)$ over semi-global adelic objects.
\begin{remark}
The mean-periodicity correspondence may be compared to automorphicity of the Hasse--Weil $L$-functions appearing in the motivic decomposition of the zeta function, for example \cite{TDAAACAROGL2}, \cite{MeAAMP}.
\end{remark}
\subsection{A Second Calculation of the Zeta Integral}
Be consistent between $\mathcal{A}$ and $\mathcal{S}$.
\newline As always, let $\mathcal{S}$ be a proper, regular model of a smooth, projective, geometrically connected curve $C$ over a number field $k$, and let $\mathcal{P}$ denote the relative projective line $\mathbb{P}^1(\mathcal{O}_k)$. Due to the renormalizing factors of the previous section, we are interested in the zeta function of the disjoint union
\[\mathcal{X}=\mathcal{S}\coprod_{i=1}^{g-1}\mathcal{P}.\]
Given any disjoint union $X=\bigcup X_i$ of schemes of finite type over $\mathbb{Z}$, one has
\[\zeta(X,s)=\prod\zeta(X_i,s),\]
therefore we have
\[\zeta(\mathcal{X},s)=\zeta(\mathcal{S},s)\zeta(\mathcal{P},s)^{g-1}.\]
Let $\mathscr{C}(\mathcal{S})$ denote a set of curves on $\mathcal{S}$, and $\mathscr{C}(\mathcal{P})$ denote a set of curves on $\mathcal{P}$. We will assume throughout that $\mathscr{C}(\mathcal{S})$ contains at least one horizontal curve, $\mathscr{C}(\mathcal{P})$ contains none, and each set includes all fibres. Let $\mathscr{C}(\mathcal{X})$ denote the union:
\[\mathscr{C}(\mathcal{X})=\mathscr{C}(\mathcal{S})\cup\mathscr{C}(\mathcal{P}).\]
We will define an analytic adelic space on $\mathcal{X}$ as the following product
\[\mathbb{A}(\mathcal{X},\mathscr{C}(\mathcal{X}))=\mathbb{A}(\mathcal{S},\mathscr{C}(\mathcal{S}))\times\overbrace{\mathbb{A}(\mathcal{P},\mathscr{C}(\mathcal{P}))\times\cdots\times\mathbb{A}(\mathcal{P},\mathscr{C}(\mathcal{P}))}^{g-1\text{ copies}}.\]
To avoid cumbersome notation, for an arithmetic surface $\mathcal{A}=\mathcal{S},\mathcal{P},\mathcal{X}$ and a set $\mathscr{C}(\mathcal{A})$ of curves on $\mathcal{A}$ we will use the notation
\[T(\mathcal{A},\mathscr{C}(\mathcal{A}))=(\mathbb{A}(\mathcal{A},\mathscr(\mathcal{A}))\times\mathbb{A}(\mathcal{A},\mathscr(\mathcal{A})))^{\times}.\]
Note that if $\mathscr{C}(\mathcal{X})$ contains only finitely many horizontal curves on $\mathcal{S}$ then the following integral converges for $s>1$:
\[\int_{T(\mathcal{X}),\mathscr{C}(\mathcal{X})}(f\coprod h)(\alpha)|\alpha|^sd\mu(\alpha),\]
where the measure on $T(\mathcal{X})$ is simply the product measure on the multiplicative adelic groups. Indeed, in the notation of the previous section, this is equal to the zeta integral
\[\zeta(\mathcal{S},\mathscr{C}(\mathcal{S}),f,h,|~|^s).\]
From now on, we will assume $\mathscr{C}(\mathcal{X})=\mathscr{C}(\mathcal{S})\cup\mathscr{C}(\mathcal{P})$ to be fixed, and simply use the notation $T(\mathcal{X})$ and 
\[\zeta(f,h,s):=\zeta(\mathcal{S},\mathscr{C}(\mathcal{S}),f,h,|~|^s).\]
Due to the presence of a horizontal curve in $\mathscr{C}(\mathcal{S})$, we have a surjective module on $T(\mathcal{X})$,
\[|~|:T(\mathcal{X})\rightarrow\mathbb{R}^{\times}_{+},\]
given as the product of modules on $\mathcal{S}$ and $\mathcal{P}$, which are modified at horizontal curves as in the previous section. $T_1(\mathcal{X})$ denotes the kernel of this module, namely
\[T_1(\mathcal{X})=\{x\in T(\mathcal{X}):|x|=1\}.\]
We may choose a splitting
\[T(\mathcal{X})\cong\mathbb{R}^{\times}_{+}\times T_1(\mathcal{X}).\]
The aim is to integrate over $T_1(\mathcal{X})$. In order to do so, we must first consider a finite subset $\mathscr{C}(\mathcal{X})^0\subset \mathscr{C}(\mathcal{X})$ containing at least one horizontal curve on $\mathcal{S}$. $\mathscr{C}(\mathcal{X})^0$ can be decomposed into a union
\[\mathscr{C}(\mathcal{S})^0\cup \mathscr{C}(\mathcal{P})^0,\]
where the first set contains only curves on $\mathcal{S}$ and the second only those on $\mathcal{P}$. For such an $\mathscr{C}(\mathcal{X})^0$, its complement will be denoted $\mathscr{C}(\mathcal{X})_0=\mathscr{C}(\mathcal{X})-\mathscr{C}(\mathcal{X})^0$. We define
\[T_{\mathscr{C}(\mathcal{X})^0}(\mathcal{X})=\prod_{y\in \mathscr{C}(\mathcal{S})^0}(\mathbb{A}(\mathcal{S},y)\times\mathbb{A}(\mathcal{S},y))^{\times}\prod_{y\in \mathscr{C}(\mathcal{P})^0}\prod_{i=1}^{g-1}(\mathbb{A}(\mathcal{P},y)\times\mathbb{A}(\mathcal{P},y))^{\times}.\]
Again, we have a surjective map
\[|~|_{\mathscr{C}(\mathcal{X})^0}:T_{\mathscr{C}(\mathcal{X})^0}(\mathcal{X})\rightarrow\mathbb{R}^{\times}_{+},\]
defined in the obvious manner. Its kernel is denoted
\[T_{\mathscr{C}(\mathcal{X})^0,1}(\mathcal{X}),\]
and we have a splitting
\[T_{\mathscr{C}(\mathcal{X})^0}(\mathcal{X})\cong\mathbb{R}^{\times}_{+}\times T_{\mathscr{C}(\mathcal{X})^0,1}(\mathcal{X}).\]
Let $p(\mathscr{C}(\mathcal{X})^0)$ denote the product of projections to one-dimensional adelic spaces as introduced in the previous section. We will fix a Haar measure on $p(T_{\mathscr{C}(\mathcal{X})^0,1}(\mathcal{X}))$ such that the Haar measure on $p(T_{\mathscr{C}(\mathcal{X})^0}(\mathcal{X}))$ is the product of this Haar measure and that on $\mathbb{R}^{\times}_{+}$, and let $\mu(T_{\mathscr{C}(\mathcal{X})^0,1})$ denote the lift of this Haar measure. For an integrable $\mathscr{F}$ on $T(\mathcal{X})$, for example $\mathscr{F}=f\coprod h$ defined by the functions in the previous section, let
\[\int_{T_1(\mathcal{X})}\mathscr{F}=\int_{T_{\mathscr{C}(\mathcal{X})^0}(\mathcal{X})}\int_{T_{\mathscr{C}(\mathcal{X})^0,1}(\mathcal{X})}\mathscr{F}(\alpha^0\gamma)d\mu(T_{S^0,1})d\mu(T_{\mathscr{C}(\mathcal{X})^0}),\]
where $\alpha^0\in T_{\mathscr{C}(\mathcal{X})^0}$ is such that
\[|\alpha^0|_{\mathscr{C}(\mathcal{X})^0}=|\alpha|^{-1}_{\mathscr{C}(\mathcal{X})^0}.\]
The integral does not depend on the choice of $\mathscr{C}(\mathcal{X})^0$, and we have the following lemma as a consequence of \cite[Lemma~43]{F3}.
\begin{lemma}
For an integrable function $\mathscr{F}$ on $T(\mathcal{X})$, we have the following
\[\int_{T(\mathcal{X})}\mathscr{F}=\int_{\mathbb{R}^{\times}_{+}}\int_{T_1(\mathcal{X})}\mathscr{F}(x\alpha)d\mu(\alpha)\frac{dx}{x}.\]
\end{lemma}
In particular, we can decompose the zeta integrals of the previous section as
\[\zeta(f,h,s)=\int_{\mathbb{R}^{\times}_{+}}\zeta_x(f\coprod h,s)\frac{dx}{x}\cdot\xi(\mathcal{P},s)^{1-g},\]
where
\[\zeta_x(f\coprod h,s)=\int_{T_1(\mathcal{X})}(f\coprod h)(m_x\alpha)|m_x\alpha|^sd\mu(\alpha).\]
This decomposition is the key to our second calculation of the zeta integral.
\begin{proposition}
Let $f,h$ be as in the previous section, then we may decompose the zeta function as a sum of the form
\[\zeta(f,h,s)=\eta(s)+\eta(2-s)+\omega(s),\]
where $\eta(s)$ absolutely converges for all $s$, and so extends to an entire function on $\mathbb{C}$.
\end{proposition}
\begin{proof}
We decompose the multiplicative group $M=\mathbb{R}^{\times}_{+}$ of positive real numbers as $M=M^{+}\cup M^{-}$, where
\[M^{\pm}=\{m\in M:\pm(|m|-1)\geq0\}.\]
We give these spaces the measure
\[\mu_{M^{\pm}}=\begin{cases}
\mu_M & \text{on }M-M\cap T_1  \\
\frac{1}{2}\mu_M & \text{on }M\cap T_1.
\end{cases}\]
The result then follows directly from
\[\zeta(f,h,s)=\int_{M^{+}}\zeta_m(f,h,s)d\mu_{M^{+}}(m)+\int_{M^{-}}\zeta_m(f,h,s)d\mu_{M^{-}}(m),\]
and
\[\omega_m(s)=\zeta_m(f,h,s)-|m|^{-2}\zeta_{m^{-1}}(f,h,s),\]
by writing
\[\eta(s)=\int_{M^{+}}\zeta_m(f,s)d\mu_{M^{+}}(m),\]
\[\omega(s)=\int_{M^{-}}\omega_m(s)d\mu_{M^{-}}(m).\]
\end{proof}
Let $\{y_i\}$ denote the complete set of horizontal curves in $\mathscr{C}(\mathcal{S})$. We will define the adelic boundary function $\mathfrak{h}(\mathcal{S},\{y_i\},\cdot):\mathbb{R}^{\times}_{+}\rightarrow\mathbb{C}$ as follows
\[\mathfrak{h}(\mathcal{S},\{y_i\},x)=\int_{T_1(\mathcal{X})}(x^2f(m_x\gamma)-f(m_x^{-1}\gamma))d\mu(\gamma),\]
where $m_x\in M\subset T(\mathcal{X})$ is a choice of representative of $x\in\mathbb{R}^{\times}_{+}$.
\newline From the above proposition we deduce the following:
\begin{corollary}\label{adelicboundary.theorem}
Let $\mathfrak{f}(\mathcal{S},\{y_i\},x)$ be the inverse Mellin transform of $\mathcal{Z}(\mathcal{S},\{y_i\},s)$, then
\[\mathfrak{h}(\mathcal{S},\{y_i\},x)=\mathfrak{f}(\mathcal{S},\{y_i\},x)x^2-\mathfrak{f}(\mathcal{S},\{y_i\},x^{-1}).\]
In particular, by Mellin inversion, and the evaluation of the zeta integrals in the previous section,
\[\mathfrak{h}(\mathcal{S},\{y_i\},x)=x^{-1/2}h(\mathcal{S},\{k(y_i)\},x).\]
\end{corollary}
In this way, we understand the boundary function $h$ of the mean-periodicity correspondence \cite{SRF} as an adelic integral. The next step is to understand the role of local-global adelic boundaries. 
\subsection{The Adelic Boundary Term}\label{adelicboundary.subsection}
Our current goal is to understand the boundary function as an integral over the topological boundary of an adelic subspace (thus motivating the terminology used throughout). This is the first step towards a verification of the mean-periodicity conjecture stated above through two-dimensional adelic duality. 
\newline Recall that, if $G$ is a topological group, then the weak topology is the weakest such that with respect to which every character of $G$ is continuous \cite[I,~2.3,~2.4]{GT}. Let $k$ be a number field and $f\in S(\mathbb{A}_k)$ be an adelic Schwartz function. In section 2 we saw the expression
\[h_f(x):=-\int_{\gamma\in\mathbb{A}^1_k/k^{\times}}\int_{\beta\in\partial k^{\times}}(f(x\gamma\beta)-x^{-1}\widehat{f}(x^{-1}\gamma\beta))d\mu(\beta)d\mu(\gamma),\]
where the boundary of $k^{\times}\hookrightarrow\mathbb{A}_k^{\times}$ is with respect to the weak topology on the locally compact topological group $\mathbb{A}_k^{\times}$. Explicitly, this is the following rational function:
\[h_f(x)=-\mu(\mathbb{A}_k^1/k^{\times})(f(0)-x^{-1}\widehat{f}(0)).\]
We will need to use a two-dimensional analogue of the inclusion $k^{\times}\hookrightarrow\mathbb{A}_k^{\times}$. 
\newline Let $\mathcal{A}$ be an arithmetic surface (in practice, it will be $\mathcal{S}$ or $\mathcal{X}=\mathcal{S}\coprod_{i=1}^{g-1}\mathcal{P}$) with function field $K$. If $y$ is a curve on $\mathcal{A}$, the field $K_y=\text{Frac}(\widehat{\mathcal{O}_y})$ is a complete discrete valuation field whose residue field is the global field $k(y)$. It is therefore neither truly local, nor truly global in nature. For all closed points $x\in y$, we have an embedding
\[K_y\hookrightarrow K_{x,y},\]
which together induce a diagonal embedding
\[K_y\hookrightarrow\prod_{x\in y}K_{x,y}.\]
For a curve $y$ on $\mathcal{A}$, let $\mathbb{B}(\mathcal{A},y)$ denote the intersection of the image of this embedding with $\mathbb{A}(\mathcal{A},y)$. Informally speaking, the counting measure on $k(y)$ lifts to an $\mathbb{R}((X))$-valued measure on $\mathbb{B}(\mathcal{A},y)$. We now make this more precise.
\begin{definition}
A basic measurable set on $\mathbb{B}(y)$ is a set of the form $t_y^{i}p_y^{-1}(A)$, where $A\subset k(y)$ is measurable with respect to the discrete counting measure\footnote{Such a set is a finite set of points.} $\mu_{k(y)}$ on $k(y)$, and $i\in\mathbb{Z}$. The measure $\mu_{\mathbb{B}(y)}$ of $t_y^{i}p_y^{-1}(A)$ is $X^i\mu_{k(y)}(A)\in\mathbb{R}((X))$.
\end{definition}
This measure can be extended by linearity. An infinite product of these measures is divergent unless, for all but finitely many $y$, each $A\subset k(y)$ contains precisely one point. Firstly we will consider finite products.
\newline Let $\mathscr{C}(\mathcal{A})^0$ be a finite set of curves on $\mathcal{A}$, and define \[\mathbb{B}(\mathcal{A},S^0)=\bigg(\prod_{y\in \mathscr{C}(\mathcal{A})^0}\mathbb{B}(\mathcal{A},y)\bigg)\cap\mathbb{A}(\mathcal{A},\mathscr{C}(\mathcal{A})^0).\]
We will integrate on $\mathbb{B}(\mathcal{A},\mathscr{C}(\mathcal{A})^0)$ with a measure induced from the product of the lifted counting measures on each $\mathbb{B}(\mathcal{A},y)$, for $y\in\mathscr{C}(\mathcal{A})^0$. Explicitly, we make the following definition: 
\begin{definition}
Let $\mathcal{A}$ be an arithmetic surface and let $\mathscr{C}(\mathcal{A})^0$ denote a finite set of curves on $\mathcal{A}$ we define the measure $\mu_{\mathbb{B}(\mathscr{C}(\mathcal{A})^0)}$ by
\[\mu_{\mathbb{B}(\mathscr{C}(\mathcal{A})^0}=\otimes_{y\in\mathscr{C}(\mathcal{A})^0}\mu_{\mathbb{B}(y)}.\]
We define the measure on $\mathbb{B}(\mathcal{A},\mathscr{C}(\mathcal{A})^0)\times\mathbb{B}(\mathcal{A},\mathscr{C}(\mathcal{A})^0)$ to be the product measure.
\end{definition}
Let $F$ be a two-dimensional local field, and let $\psi$ be a choice of character such that all continuous characters of $F$ are of the form
\[\psi_a:\alpha\mapsto\psi(a\alpha),\]
for $a\in F$. For an integrable function $f$ on $F$, the Fourier transform $\mathcal{F}(f)$ with respect to $\psi$ is defined by
\[\mathcal{F}(f)(\beta)=\int_Ff(\alpha)\psi(\alpha\beta)d\alpha.\]
In particular, this applies to fields of the form $K_{x,y}$ and we denote the Fourier transform on these fields by $\mathcal{F}_{x,y}$. For any integrable function $f_y$ on $\mathbb{A}(\mathcal{A},y)$, we may define
\[\mathcal{F}_y(f_y)=\otimes_{x\in y}\mathcal{F}_{x,y}(f_{x,y}).\]
By \cite[Proposition~32]{F3}, we have a ``summation formula''\footnote{The semi-global adelic object $\mathbb{B}$ is discrete in $\mathbb{A}$.}:
\begin{proposition}
Let $\mathscr{C}(\mathcal{A})^0$ be a finite set of curves on an arithmetic surface $\mathcal{A}$, and $f$ be an integrable function on $\mathbb{B}(\mathcal{A},\mathscr{C}(\mathcal{A})^0)\times\mathbb{B}(\mathcal{A},\mathscr{C}(\mathcal{A})^0)$ then, for all $\alpha\in\mathbb{A}(\mathcal{A},\mathscr{C}(\mathcal{A})^0)\times\mathbb{A}(\mathcal{A},\mathscr{C}(\mathcal{A})^0)$,
\[\int f(\alpha\beta)d\mu_{\mathbb{B}(\mathcal{A},\mathscr{C}(\mathcal{A})^0)\times\mathbb{B}(\mathcal{A},\mathscr{C}(\mathcal{A})^0)}(\beta)=\frac{1}{|\alpha|}\int\mathcal{F}(f)(\alpha^{-1}\beta)d\mu_{\mathbb{B}(\mathcal{A},\mathscr{C}(\mathcal{A})^0)\times\mathbb{B}(\mathcal{A},\mathscr{C}(\mathcal{A})^0)}(\beta).\]
\end{proposition}
We want a multiplicative analogue of this statement. Let $y$ be a curve on $\mathcal{A}$, we introduce the notation \[T_0(\mathcal{A},y)=\mathbb{B}(\mathcal{A},y)^{\times}\times\mathbb{B}(\mathcal{A},y)^{\times}\subset T(\mathcal{A},y).\]
To integrate on this space we make the following definitions. 
\begin{definition}
Let $y$ be a curve on an arithmetic surface $\mathcal{A}$, and let $\mu_{k(y)^{\times}}$ denote the discrete counting measure on the multiplicative group $k(y)^{\times}$ of the function field $k(y)$. A basic measurable $\mu_{\mathbb{B}(\mathcal{A},y)^{\times}}$ on $\mathbb{B}(\mathcal{A},y)^{\times}$ is of the form $t_y^ip_y^{-1}(A)$, where $A$ is a $\mu_{k(y)^{\times}}$-measurable set of $k(y)^{\times}$, and its measure is $X^i\mu_{k(y)}(A)\in\mathbb{R}((X))$. We define
\[\mu_{T_0(\mathcal{A},y)}=\mu_{\mathbb{B}(\mathcal{A},y)^{\times}}\otimes\mu_{\mathbb{B}(\mathcal{A},y)^{\times}}\]
These measure can be extended to finite unions by linearity.
\end{definition}
For a subset $\mathscr{C}(\mathcal{A})^0$ of finitely many curves on $\mathcal{A}$, let
\[T_0(\mathcal{A},\mathscr{C}(\mathcal{A})^0)=\prod_{y\in\mathscr{C}(\mathcal{A})^0}T_0(\mathcal{A},y)\subset T(\mathcal{A},\mathscr{C}(\mathcal{A})^0).\]
On $T_0(\mathcal{A},\mathscr{C}(\mathcal{A})^0)$, we introduce the measure
\[\mu_{T_0(\mathcal{A},\mathscr{C}(\mathcal{A})^0)}=\prod_{y\in S^0}(q_y-1)^{-2}\mu_{T_0(\mathcal{A},y)}.\]
Finally, let $\mathscr{C}(\mathcal{A})$ be a set containing all fibres and finitely many horizontal curves, we define
\[T_0(\mathcal{A},\mathscr{C}(\mathcal{A}))=\prod_{y\in\mathscr{C}(\mathcal{A})}T_0(\mathcal{A},y).\]
On this space, we integrate using the following rule
\[\int_{T_0(\mathcal{A},\mathscr{C}(\mathcal{A}))}\mathscr{F}=\lim_{\mathscr{C}(\mathcal{A})^0\subset\mathscr{C}(\mathcal{A})}\int_{T_0(\mathcal{A},\mathscr{C}(\mathcal{A})^0)}\mathscr{F},\]
Integrable functions are those such that this limit is finite. Overall, we have a filtration
\[T_0(\mathcal{A},S)\subset T_1(\mathcal{A},S)\subset T(\mathcal{A},S).\]
In \cite[Section~3.5]{F3}, Fesenko introduces a measure on the quotient such that, for an integrable function $g$ on $T(\mathcal{A},\mathscr{C}(\mathcal{A}))$:
\[\int_{T_1(\mathcal{A},\mathscr{C}(\mathcal{A}))}g=\int_{T_1(\mathcal{A},\mathscr{C}(\mathcal{A}))/T_0(\mathcal{A},\mathscr{C}(\mathcal{A}))}\int_{T_0(\mathcal{A},\mathscr{C}(\mathcal{A}))}g(\gamma\beta)d\mu(\beta) d\mu(\gamma)\]
Let $\mathscr{C}(\mathcal{A})^0\subset\mathscr{C}(\mathcal{A})$ be a finite subset. We endow $\mathbb{A}(\mathcal{A},\mathscr{C}(\mathcal{A})^0)\times\mathbb{A}(\mathcal{A},\mathscr{C}(\mathcal{A})^0)$ with the weakest topology such that each character lifted by $p$ is continuous. With respect to this topology, we will call the boundary $\partial T_0(\mathcal{A},\mathscr{C}(\mathcal{A})^0)$ of $T_0(\mathcal{A},\mathscr{C}(\mathcal{A})^0)\subset\mathbb{A}(\mathcal{A},\mathscr{C}(\mathcal{A})^0)\times\mathbb{A}(\mathcal{A},\mathscr{C}(\mathcal{A})^0)$ the ``weak boundary''. Note that this is a measurable subset of $\mathbb{B}(\mathcal{A},\mathscr{C}(\mathcal{A})^0)\times\mathbb{B}(\mathcal{A},\mathscr{C}(\mathcal{A})^0)$, and we may define
\[\int_{\partial T_0(\mathcal{A},\mathscr{C}(\mathcal{A})^0)}g=d(\mathscr{C}(\mathcal{A})^0)\int_{\mathbb{B}(\mathcal{A},\mathscr{C}(\mathcal{A})^0)\times\mathbb{B}(\mathcal{A},\mathscr{C}(\mathcal{A})^0}g\cdot\text{char}_{\partial T_0(\mathcal{A},\mathscr{C}(\mathcal{A})^0)}d\mu_{\mathbb{B}(\mathcal{A},\mathscr{C}(\mathcal{A})^0)\times\mathbb{B}(\mathcal{A},\mathscr{C}(\mathcal{A})^0},\]
where
\[d(\mathscr{C}(\mathcal{A})^0)=\prod_{y\in\mathscr{C}(\mathcal{A})^0}(q_y-1)^{-2}.\]
We are interested in an inductive limit of these weak boundaries:
\[\partial T_0(\mathcal{A},\mathscr{C}(\mathcal{A}))=\bigcup_{S_0\subset S}T_0(\mathcal{A},\mathscr{C}(\mathcal{A})^0),\]
where the union runs over finite subsets $\mathscr{C}(\mathcal{A})^0\subset\mathscr{C}(\mathcal{A})$. If $g$ is an integrable function on $\mathbb{A}(\mathcal{A},\mathscr{C}(\mathcal{A}))\times\mathbb{A}(\mathcal{A},\mathscr{C}(\mathcal{A}))$, then one defines
\[\int_{\partial T_0(\mathcal{A},\mathscr{C}(\mathcal{A}))}g=\lim_{\mathscr{C}(\mathcal{A})^0\subset \mathscr{C}(\mathcal{A})}\int_{\partial T_0(\mathcal{A},\mathscr{C}(\mathcal{A})^0)}g.\]
The limit expressing the integral is only finite in exceptional circumstances. One such case is expressed by the so-called ``two-dimensional theta formula''.
\begin{theorem}\label{2DTheta.theorem}
Let $\mathcal{A}$ be an arithmetic surface, $S$ denote a set of curves on $\mathcal{A}$, and $f$ be an integrable function on $\mathbb{A}(\mathcal{A},S)\times\mathbb{A}(\mathcal{A},S)$, then
\[\int_{T_0(\mathcal{A})}(f(\alpha\beta)-\mathcal{F}(f_{\alpha})(\beta))d\mu(\beta)=\int_{\partial T_0(\mathcal{A})}(\mathcal{F}(f_{\alpha}(\beta))-f(\alpha\beta))d\mu(\beta).\]
\end{theorem}
One applies this result to $\mathcal{A}=\mathcal{X}$ and $g=f\coprod h$. Consequently, one obtains the boundary integral contribution to $h(\mathcal{S},\{k(y_i)\},x)$. It transpires that this boundary integral knows much about the analytic properties of zeta.
\begin{corollary}\label{boundaryintegral.corollary}
Let $\{y_i\}$ be the set of horizontal curves on $\mathcal{S}$ in $\mathscr{C}(\mathcal{X})$. We may decompose the boundary integral as follows
\[\mathfrak{h}(\mathcal{S},\{y_i\},x)=\mathfrak{h}_1(\mathcal{S},\{y_i\},x)+\mathfrak{h}_2(\mathcal{S},\{y_i\},x),\]
where
\[\mathfrak{h}_1(x)=\int_{T_1(\mathcal{X},\mathscr{C}(\mathcal{X}))}(|\alpha^{-1}|-1)f(m_x^{-1}\alpha^{-1})d\mu(\alpha)\]
\[\mathfrak{h}_2(x)=x^2\int_{[T_1/T_0](\mathcal{X},\mathscr{C}(\mathcal{X}))}\int_{\partial T_0(\mathcal{X},\mathscr{C}(\mathcal{X}))}(|m_x\gamma|^{-1}f(m_x^{-1}\nu^{-1}\gamma^{-1}\beta)-f(m_x\gamma\beta))d\mu(\beta)d\mu(\gamma),\]
and $m_x$ are lifts of $x\in\mathbb{R}^{\times}_{+}$ to $T(\mathcal{X})$, and $\nu$ is as in \cite[Section~3.4]{F3}.
\end{corollary}
For elliptic curves, this result was first deduced by Fesenko. An analogous decomposition, not involving adelic integrals, appeared in \cite[Remark~5.11]{SRF}.
\subsection*{Acknowledgement}
The author was supported by the Heilbronn Institute for Mathematical Research. Numerous stimulating conversations on the subject of this paper were had with I. Fesenko, to whom the author is grateful.
%\newline It is a pleasure to dedicate this paper to Professor S.V. Vostokov. His explicit symbol plays a fundamental role in the class field theory of higher dimensional local fields and the study of their topological $K$-groups, the latter of which underlie the local factors of the two-dimensional zeta integrals studied in this paper.
\bibliography{studygroup}

\begin{thebibliography}{10}

\bibitem{RAA}
A.~Beilinson.
\newblock Residues and adeles.
\newblock {\em Funct. Anal. Appl.}, 14:34--35, 1980.

\bibitem{DRCACOC}
S.~Bloch.
\newblock De {R}ham cohomology and conductors of curves.
\newblock {\em Duke Math. J.}, 54:295--308, 1987.

\bibitem{GT}
N.~Bourbaki.
\newblock {\em General Topology}.
\newblock Hermann, 1966.

\bibitem{CLFASROSC}
I.~I. Brouw and S.~Wewers.
\newblock Computing {L}-{F}unctions and semistable reduction of superelliptic
  curves.
\newblock arXiv:1211.4459v1, 2012.

\bibitem{TIOTSOCOGG}
P.~Deligne and D.~Mumford.
\newblock The irreducibility of the space of curves of given genus.
\newblock {\em Inst. Hautes Études Sci. Publ. Math.}, 36:75--109, 1969.

\bibitem{SRF}
I.~Fesenko, G.~Ricotta, and M.~Suzuki.
\newblock {M}ean-{P}eriodicity and zeta functions.
\newblock {\em Ann. L'Inst. Fourier}, 12:1819--1887, 2012.

\bibitem{F1}
I.B. Fesenko.
\newblock Analysis on arithmetic schemes {I}.
\newblock {\em Documenta Mathematica}, 8:261 -- 284, 2003.

\bibitem{F2}
I.B. Fesenko.
\newblock Measure, integration and elements of harmonic analysis on generalised
  loop spaces.
\newblock {\em Proceed. St Petersburg Math. Soc.}, 12:179-- 199, 2006.

\bibitem{F4}
I.B. Fesenko.
\newblock Adelic approach to the zeta function of arithmetic schemes in
  dimension two.
\newblock {\em Moscow Math. J.}, 8:273 -- 317, 2008.

\bibitem{F3}
I.B. Fesenko.
\newblock Analysis on arithmetic schemes {I}{I}.
\newblock {\em J. K-theory}, 5:437--557, 2010.

\bibitem{GAATRRFOCOS}
I.B. Fesenko.
\newblock Geometric adeles and the {R}iemann-{R}och theorem for 1-cycles on
  surfaces.
\newblock {\em Max-Planck-Institut fur Mathematik Preprint Series}, 36:1--10,
  2012.

\bibitem{IHLF}
I.B. Fesenko and M.~Kurihara, editors.
\newblock {\em Invitation to Higher Local Fields}.
\newblock Geometry and Topology Monographs 3, 2000.

\bibitem{FV}
I.B. Fesenko and S.V. Vostokov.
\newblock {\em Local Fields and Their Extensions}.
\newblock American Mathematical Society, 2 edition, 2002.

\bibitem{IIVF}
E.~Hrushovski and D.~Kazhdan.
\newblock Integration in valued fields.
\newblock {\em Progr. Math.}, 253:261--405, 2006.

\bibitem{OTPBAFS}
A.~Huber.
\newblock On the {P}arshin-{B}eilinson adeles for schemes.
\newblock {\em Abh. Math. Sem. Univ. Hamburg}, 61:249--273, 1991.

\bibitem{SHAOSL2O2DLF}
H.H. Kim and K.-H. Lee.
\newblock Spherical {H}ecke algebras of {SL}2 over 2-dimensional local fields.
\newblock {\em American Journal of Mathematics}, pages 1381--1399, 2004.

\bibitem{L}
Q~Liu.
\newblock {\em Algebraic Geometry and Arithmetic Curves}.
\newblock Oxford University Press, 2002.

\bibitem{OAROTICGRTPAZOLF}
R.~Meyer.
\newblock On a representation of the idele class group related to {P}rimes and
  zeros of {L}-{F}unctions.
\newblock {\em Duke Mathe. J.}, 127:519--595, 2005.

\bibitem{FTANLCOVOATDLF}
M.T. Morrow.
\newblock Fubini's theorem and non-linear change of variables over a
  two-dimensional local field.
\newblock arXiv:0712.2177v3, 2007.

\bibitem{Mor2}
M.T. Morrow.
\newblock Integration on valuation fields over local fields.
\newblock {\em The Tokyo Journal of Mathematics}, 33:235 -- 281, 2010.

\bibitem{AITHDLFAA}
M.T. Morrow.
\newblock An introduction to higher dimensional local fields and adeles.
\newblock arXiv:1204.0586, 2012.

\bibitem{MeAAMP}
T.~D. Oliver.
\newblock Automorphicity and mean-periodicity.
\newblock arXiv:1307.6706, 2013.

\bibitem{P1}
A.~N. Parshin.
\newblock On the arithmetic of two-dimensional schemes, repartitions and
  residues.
\newblock {\em Math. Izv.}, 10:695 -- 747, 1976.

\bibitem{P3}
A.~N. Parshin.
\newblock Chern classes, adeles and {L}-{F}unctions.
\newblock {\em J. fuer die reine und angew. Math.}, 341:174 -- 192, 1983.

\bibitem{CDATNFOAS}
T.~Saito.
\newblock Conductor, discriminant, and the {N}oether formula of arithmetic
  surfaces.
\newblock {\em Duke Math. J.}, 57:151--173, 1988.

\bibitem{ZALF}
J.P. Serre.
\newblock Zeta and {L}-{F}unctions.
\newblock {\em Arithmetical Algebraic Geometry (Proc. Conf. Purdue Univ.)},
  Harper \& Row, New York:82--92, 1963.

\bibitem{TDAAACAROGL2}
M.~Suzuki.
\newblock Two dimensional adelic analysis and cuspidal automorphic
  representations of {G}{L}(2).
\newblock {\em Progress in Mathematics; Multiple Dirichlet Series, L-functions
  and Automorphic Forms}, 300:339--361, 2012.

\bibitem{T1}
J.~Tate.
\newblock {\em Fourier analysis in number fields and {H}ecke's zeta functions}.
\newblock PhD thesis, Princeton University, 1950.

\bibitem{FZED}
A.~Weil.
\newblock Fonction zeta et distributions.
\newblock {\em Seminaire Bourbaki 9}, Exp. No. 312:523--531, 1965/6.

\bibitem{BNT}
A.~Weil.
\newblock {\em Basic Number Theory}.
\newblock Springer-Verlang, 1974.

\end{thebibliography}
\bibliographystyle{plain}
\end{document}